\documentclass[11pt]{amsart}
\usepackage{a4}
\usepackage{amssymb}
\usepackage{array}
\usepackage{booktabs}
\usepackage{multirow}
\usepackage{hhline}
\usepackage{mathabx}

\usepackage {amsfonts}
\usepackage{amsthm}
\usepackage{latexsym}
\usepackage{amsmath}
\usepackage{xcolor}
\usepackage[all]{xy}
\usepackage{comment}

\global\long\def\tp{\mathop{\xymatrix{*+<.7ex>[o][F-]{\scriptstyle \top}}
 } }
 
\global\long\def\bp{\mathop{\xymatrix{*+<.7ex>[o][F-]{\scriptstyle \bot}}
 } }

\newtheorem{thm}{Theorem}

\newtheorem{lem}[thm]{Lemma}
\newtheorem{prop}[thm]{Proposition}
\newtheorem{defn}[thm]{Definition}
\newtheorem{ex}[thm]{Example}

\newtheorem{rem}[thm]{Remark}

\numberwithin{thm}{section}
\numberwithin{equation}{section}
\newcommand{\Comp}{\mathbb C}
\newcommand{\Real}{\mathbb R}

\newcommand{\norm}[1]{\left\Vert#1\right\Vert}
\newcommand{\abs}[1]{\left\vert#1\right\vert}

\newcommand{\IrrQG}{{\rm Irr}(\mathbb{G})}
\newcommand{\DcoP}{\widehat{\Delta}}
\newcommand{\PolQG}{{\rm Pol}(\mathbb{G})}
\newcommand{\QG}{\mathbb{G}}
\newcommand{\QDG}{\widehat{\mathbb{G}}}

\newcommand{\QH}{\mathbb{H}}
\newcommand{\QDH}{\widehat{\mathbb{H}}}

\newcommand{\la}{\langle}
\newcommand{\ra}{\rangle}

\newcommand{\A}{\mathcal{A}}

\newcommand{\B}{\mathcal{B}}

\newcommand{\F}{\mathcal{F}}

\newcommand{\Gb}{\mathbb{G}}

\newcommand{\I}{\mathcal{I}}

\newcommand{\tor}{\mathbb{T}}

\newcommand{\z}{\mathbb{Z}}

\newcommand{\prt}{\widehat{\otimes}}

\begin{document}

\title[Beurling-Fourier algebras of compact quantum groups]{Beurling-Fourier algebras of compact quantum groups: characters and finite dimensional representations}
\author{Uwe Franz}
\address{D\'epartement de math\'ematiques de Besan\c{c}on,
Universit\'e de Bourgogne Franche-Comt\'e 16, route de Gray, 25 030
Besan\c{c}on cedex, France}
\email{uwe.franz@univ-fcomte.fr}
\urladdr{http://lmb.univ-fcomte.fr/uwe-franz}

\author{Hun Hee Lee}
\address{Department of Mathematical Sciences and the Research Institute of Mathematics, Seoul National University, Gwanak-ro 1, Gwanak-gu, Seoul 08826, Republic of Korea}
\email{hunheelee@snu.ac.kr}
\urladdr{http://http://www.math.snu.ac.kr/~hhlee/}

\thanks{UF was supported by the French "Investissements d'Avenir" program, project ISITE-BFC (contract ANR-15-IDEX-03). HHLee was supported by the Basic Science Research Program through the National Research Foundation of Korea (NRF) Grant NRF-2017R1E1A1A03070510 and the National Research Foundation of Korea (NRF) Grant funded by the Korean Government (MSIT) (Grant No.2017R1A5A1015626).} 

\begin{abstract}
In this paper we study weighted versions of Fourier algebras of compact quantum groups. We focus on the spectral aspects of these Banach algebras in two different ways. We first investigate their Gelfand spectrum, which shows a connection to the maximal classical closed subgroup and its complexification. Second, we study specific finite dimensional representations coming from the complexification of the underlying quantum group. We demonstrate that the weighted Fourier algebras can detect the complexification structure in the special case of $SU_q(2)$, whose complexification is the quantum Lorentz group $SL_q(2,\Comp)$.
\end{abstract}
\maketitle


\section{Introduction}
\label{sec-intro}

To a locally compact group $G$ one can associate a commutative Banach algebra $A(G)$ called the Fourier algebra. As a Banach space $A(G)$ is the (unique) predual of $VN(G)$, the group von Neumann algebra of $G$ and there is a natural inclusion $A(G) \hookrightarrow C_0(G)$, where the algebra multiplication on $C_0(G)$, i.e. the pointwise multiplication, is inherited by $A(G)$. The Fourier algebra $A(G)$ contains enough information to determine the underlying group $G$. In particular, we have the Gelfand spectrum ${\rm Spec}\,A(G) \cong G$ as topological spaces. Recently, there have been several studies on the weighted versions of Fourier algebras (\cite{GLSS, LS, LST}) under the name of Beurling-Fourier algebras. Especially, \cite{LST} deals with the spectral theory of Beurling-Fourier algebras on compact groups.
For a compact group $G$ a function $w: {\rm Irr}(G) \to (0,\infty)$ is called a {\it weight} if $$w(t) \le w(s)w(s'),\;\; \forall s, s',t\in {\rm Irr}(G),\; t\subseteq s\tp s'.$$
Here, ${\rm Irr}(G)$ is the equivalence class of all irreducible unitary representations of $G$, and $t\subseteq s\tp s'$ means that $t$ is a subrepresentation of the tensor product $s \tp s'$, which will be clarified later. A weight $w$ gives rise to a weighted space $A(G,w)$ given by $$A(G,w) := \{f\in C(G) |\; ||f||_{A(G,w)}:=\sum_{s\in {\rm Irr}(G)}n_sw(s)||\widehat{f}(s)||_1<\infty\},$$
where $n_s$ refers to the dimension of the representation $s$, $||X||_1$ is the trace norm of a square matrix $X$	 and $\widehat{f}(s)$ is the Fourier coefficient of $f$ at $s$ given by $\widehat{f}(s) = \int_G f(g) \bar{s}(g)dg$, where $\bar{s}$ is the conjugate representation of $s$. It has been shown that $A(G,w)$ is still a commutative Banach algebra under the pointwise multiplication. It is natural to be interested in the Gelfand spectrum ${\rm Spec}\,A(G,w)$ and it has been proved to be closely related to the complexification of $G$ (\cite{LST}). This connection is based on the fact that ${\rm Pol}(G)$, the algebra of coefficient functions of irreducible unitary representations of $G$, is sitting densely inside $A(G,w)$ regardless of the choice of $w$.
Then, we have ${\rm Spec}\,A(G,w) \subseteq {\rm Spec}\,{\rm Pol}(G)$, the set of all non-zero multiplicative functionals on ${\rm Pol}(G)$. The latter object ${\rm Spec}\,{\rm Pol}(G)$ can be regarded as a subset of the algebraic dual ${\rm Pol}(G)^\dagger$, which allows a natural locally compact group structure. In \cite{McK} K.\ McKennon defined ${\rm Spec}\, {\rm Pol}(G)$ to be an abstract complexification of $G$, which we denote by $G_\Comp$, following the same construction of a universal complexification model for a compact connected Lie group $G$ by Chevalley (\cite[III. 8]{BtD}).

From the above discussion we could argue that ${\rm Spec}\,A(G,w)$ is always realized inside the complexification $G_\Comp$ of $G$. Moreover, in many cases, the union of ${\rm Spec}\,A(G,\omega)$ becomes the whole complexification. For example, we consider an abelian example of $G = \tor$ and $w_\beta: \z \to (0,\infty),\; n\mapsto \beta^{|n|}$, $\beta \ge 1$. It is well-known (\cite[Proposition 2.8.8]{Kan}) that
	$$\text{Spec}\,A(\tor, w_\beta) \cong \{c\in \Comp: \frac{1}{\beta} \le |c| \le \beta\}$$
homeomorphically. The right hand side of the above identification actually comes from the observation $\text{Spec}\,A(\tor, w_\beta) \subseteq \text{Spec}\,{\rm Pol}(\tor) = \tor_\Comp = \Comp^*$, which is independent of the choice of weight $w_\beta$. This universal point of view allows us to compare various spectra $\text{Spec}\,A(\tor, w_\beta)$ in the same framework and we have
	$$\bigcup_{\beta \ge 1} \text{Spec}\,A(\tor, w_\beta) \cong \Comp^* = \tor_\Comp.$$
The simplest non-abelian example would be $G = SU(2)$ with $w_\beta: {\rm Irr}(SU(2)) = \frac{1}{2}\z_+ \to (0,\infty),\; s \mapsto \beta^{2s}$, $\beta\ge 1$. Then it has been shown (\cite[Example 4.5]{LST}) that
	$$\text{Spec}\,A(SU(2), w_\beta) \cong \{U\begin{bmatrix}c &0 \\ 0 & c^{-1}\end{bmatrix}V: U,V \in SU(2),\; \frac{1}{\beta} \le |c| \le \beta\},$$
homeomorphically. As before the right hand side of the above identification actually comes from the observation $\text{Spec}\,A(SU(2), w_\beta) \subseteq \text{Spec}\,{\rm Pol}(SU(2)) = SU(2)_\Comp = SL_2(\Comp)$, which is independent of the choice of weight $w_\beta$, so that we have
	$$\bigcup_{\beta \ge 1} \text{Spec}\,A(SU(2), w_\beta) \cong SL_2(\Comp) = SU(2)_\Comp.$$

In this paper we would like to continue the above line of research in the quantum context. For a compact quantum group $\QG$ we will first define Beurling-Fourier algebras $A(\QG,w)$ extending the case of compact groups. Since the resulting Banach algebra $A(\QG,w)$ is non-commutative in general we have two possible directions to pursue. First, we can directly follow the classical situation by investigating the Gelfand spectrum ${\rm Spec}\,A(\QG,w)$.
It turns out that ${\rm Spec}\,A(\QG,w)$ contains the information about the complexification of the maximal classical closed subgroup of $\QG$. One of the main differences here is that the algebra $A(\QG,w)$ is in general non-commutative whilst $A(G,w)$ is always commutative. It is quite natural to expect that the Gelfand spectrum of non-commutative Banach algebra will provide only a limited amount of information. This limitation leads us to look for other aspects of $A(\QG,w)$. Let us go back to the classical situation of a compact Lie group $G$. We already know that
$$\text{Spec}\,{\rm Pol}(G) \cong G_\Comp \cong {\rm Spec}\,C_0(G_\Comp)$$
as locally compact spaces. Note further that the Gelfand spectrum ${\rm Spec}\,C_0(G_\Comp)$ is the same as the $C^*$-algebra spectrum ${\rm sp}\,C_0(G_\Comp)$, which is the set of equivalence classes of all irreducible $*$-representation $\pi : C_0(G_\Comp) \to B(H)$ for some Hilbert space $H$ equipped with the canonical topology. The latter provides us with a possible quantum extension of the story. More precisely, if we begin with an irreducible $*$-representation $\pi : C_0(G_\Comp) \to B(H)$ then $H$ must be 1-dimensional and there is a point $x\in G_\Comp$ such that $\pi = \varphi_x$, the evaluation functional on $C_0(G_\Comp)$ at $x$.
The functional $\varphi_x$ can be transferred to the setting of $H(G_\Comp)$, the space of holomorphic functions on $G_\Comp$ and from the standard Lie theory it is well-known that ${\rm Pol}(G)$ can be embedded in $H(G_\Comp)$ in a canonical way. We would like to repeat the same procedure in the quantum setting, namely, we consider the {\it complexification} of the compact quantum group $\QG$. It is commonly believed that the Drinfeld-double construction gives us an object we should regard as {\it the} complexification in many genuine quantum group cases, but we will focus on the most well-understood example of $SL_q(2,\Comp)$, which is the complexification of $SU_q(2)$, $-1<q\ne 0<1$.

This paper is organized as follows. In section \ref{sec-prelim} we collect basic materials we need in this paper including the definition of $SL_q(2,\Comp)$. In section \ref{sec-BF-alg-Def} we define Beurling-Fourier algebras on compact quantum groups using central weights.
In section \ref{sec-characters} we investigate the Gelfand spectrum of the Beurling-Fourier algebras on compact quantum groups and clarify its connection to the complexification of the maximal classical closed subgroup. In the final section \ref{sec-finite-dim-rep} we begin with more details of $SL_q(2,\Comp)$ and the associated Iwasawa decomposition. We, then, consider irreducible $*$-representations on $C_0(SL_q(2, \Comp))$ and classify them according to their (complete) boundedness on Beurling-Fourier algebras.

\section{Preliminaries}\label{sec-prelim}

\subsection{Compact quantum groups}
In this paper $\otimes$ refers to the minimal tensor product of $C^*$-algebras and operator spaces, whilst $\odot$ denotes the algebraic tensor product.

Recall that a compact quantum group $\QG$ is a pair $(C(\QG), \Delta)$, where $C(\QG)$ is a unital $C^*$-algebra considered as the algebra of ``continuous functions on $\QG$'' and $\Delta:C(\QG) \to C(\QG) \otimes C(\QG)$ is a unital, injective $*$-homomorphism which is coassociative, i.e.
\[ (\Delta \otimes id)\circ \Delta = (id \otimes \Delta)
\circ\Delta
\]
and
\[
\overline{{\rm Lin}}((1\otimes C(\QG))\Delta(C(\QG)) ) = \overline{{\rm Lin}}((C(\QG) \otimes
1)\Delta(C(\QG)) ) = C(\QG) \otimes C(\QG).
\]
In this case there exists a unique state $h \in C(\QG)^*$ called the \emph{Haar state} such that for all $a \in C(\QG)$
\[
(h \otimes id)\circ  \Delta (a) = ( id \otimes h)\circ  \Delta (a) = h(a) 1.
\]

An element $u =(u_{kl})_{1\le k,l\le n}\in M_n(C(\QG)) = M_n \otimes C(\QG)$ is called an {\emph{$n$-dimensional representation of}} $\QG$ if we have 	
	$$\Delta(u_{kl}) =
\sum_{j=1}^n u_{kj} \otimes u_{jl}\;\; \text{for all} \;\; 1\le k,l \le n.$$
A representation $u$ is said to be \emph{unitary}, if $u$ is unitary, and \emph{irreducible}, if the only matrices $T\in M_n$ with
$Tu=uT$ are multiples of the identity matrix. Two unitary  representations $u,v\in M_n(C(\QG))$ are called
\emph{(unitarily) equivalent} (we write $u\cong v$), if there exists a unitary matrix $U\in M_n(\mathbb{C})$ such that $Uu=v\,U$. For two representations $u\in M_n(C(\QG))$, $v\in M_m(C(\QG))$ we have the concept of their direct sum $u\oplus v \in (M_n \oplus M_m)\otimes C(\QG) \subseteq M_{n+m}(C(\QG))$. It is well-known that any finite demensional representation decomposes into a direct sum of irreducible representations. If $u$ is an irreducible unitary representation appearing in the decomposition of another unitary representation $v$, then we denote it by
	$$u \subseteq v.$$
We also define their tensor product $u \tp v$ by
	$$u \tp v = u_{13}v_{23} \in M_n\otimes M_m \otimes C(\QG) = M_{nm}(C(\QG))$$
following the leg notation. We recall a related notation $u \bp v$. Let $u\in M_n(\A)$ and $v\in M_n(\B)$ for some $C^*$-algebras $\A$ and $\B$. Then, we define
	$$u \bp v = u_{12}v_{13} \in M_n(\A\otimes \B).$$
For $A\in M_n(B(H))$, $B\in M_n(B(K))$, $C\in M_m(B(H))$ and $D\in M_m(B(K))$ we clearly have
	\begin{equation}\label{eq-tensor-top-bot}
		(A \bp B) \tp (C \bp D) = (A \tp C) \bp (B \tp D)
	\end{equation}
in $M_{nm}(B(H)\otimes B(K))$.
	
An important feature of compact quantum groups is the existence of the dense $*$-subalgebra $\text{Pol}(\QG)$, which is in fact a Hopf $*$-algebra with the coproduct
$\Delta|_{\text{Pol}(\QG)}$ -- so for example $\Delta: \text{Pol}(\QG) \to \text{Pol}(\QG) \odot \text{Pol}(\QG)$. Fix a complete family $(u^{(s)})_{s\in \IrrQG}$ of mutually inequivalent irreducible unitary representations of $\QG$, then $\{u^{(s)}_{k\ell}: s\in \text{Irr}(\QG), 1\le k,\ell\le n_s\}$ (where $n_s$ denotes the dimension of $u^{(s)}$) is a linear basis of $\text{Pol}(\QG)$, cf.\ \cite[Proposition 5.1]{woronowicz98}. For each $s \in \text{Irr}(\QG)$ there exists a unique positive matrix $Q_s \in GL(n_s)$ such that $\text{Tr}(Q_s) = \text{Tr}(Q_s^{-1})=:d_s\geq n_s$ and we have for all $s,t \in \text{Irr}(\QG)$
	\begin{equation}\label{eq-ortho}
	h\left(u_{ij}^{(s)} (u_{kl}^{(t)})^* \right) = \delta_{s t} \delta_{ik} \frac {(Q_s)_{l,j}}{d_s},\; \;h\left((u_{ij}^{(s)})^* u_{kl}^{(t)} \right) = \delta_{st} \delta_{jl} \frac {(Q_s^{-1})_{k,i}}{d_s}.
	\end{equation}
The number $d_s$ is called the \emph{quantum dimension} of the representation $u^{(s)}$. 

A compact quantum group always allows the reduced and the universal versions $C_r(\QG)$ and $C_u(\QG)$ as follows. Let $\pi : C(\QG) \to B(L^2(\QG))$ be the GNS representation of $C(\QG)$ with respect to the Haar state $h$. Then $C_r(\QG)$ is given by $\overline{\pi(C(\QG))} \subseteq B(L^2(\QG))$. The universal version $C_u(\QG)$ is the $C^*$-algebra completion of $\PolQG$ with respect the $C^*$-norm $||\cdot||_u$ given by
	$$||a||_u = \sup_{\pi} ||\pi(a)||,$$
where the supremum runs over all unital $*$-homomorphism $\pi: \PolQG \to B(H)$ for some Hilbert space $H$.	Both of the versions are equiped with natural compact quantum group structures with the transferred co-multiplications $\Delta_r$ and $\Delta_u$. From the universality of $C_u(\QG)$ we have a  natural onto $*$-homomorphism $\rho: C_u(\QG) \to C_r(\QG)$ extending the identity map on $\PolQG$. 

\subsection{Discrete quantum groups and the Fourier algebra of compact quantum groups}\label{subsec-discrete}

To a compact quantum group $\QG$ we can associate its discrete dual quantum group $\widehat{\QG}$. We begin with the $C^*$-algebra $c_0(\widehat{\QG})$ and the von Neumann algebra $\ell^\infty(\widehat{\QG}) = VN(\QG)$ given by
	$$c_0(\widehat{\QG}) = c_0\,\text{-}\bigoplus_{s\in \text{Irr}(\QG)} M_{n_s},\;\; \ell^\infty(\widehat{\QG}) = \ell^\infty \text{-}\bigoplus_{s\in \text{Irr}(\QG)} M_{n_s}.$$
There is a natural co-multiplication $\DcoP: \ell^\infty(\widehat{\QG}) \to \ell^\infty(\widehat{\QG}) \bar{\otimes}\, \ell^\infty(\widehat{\QG})$, which is normal, given by
	$$(\DcoP \otimes id)\mathbb{U} = \mathbb{U}_{12}\mathbb{U}_{13},$$
where $\mathbb{U} = \oplus_{s\in \text{Irr}(\QG)} u^{(s)}$ is the multiplicative unitary in the multiplier algebra $M(c_0(\widehat{\QG}) \otimes C(\QG))$. Here, $\bar{\otimes}$ is the spatial tensor product of von Neumann algebras. The algebra $\ell^\infty(\widehat{\QG})$ is equipped with the left and the right Haar weight $\widehat{h}_L$ and $\widehat{h}_R$ given by
	$$\widehat{h}_L(X) = \sum_{s\in \IrrQG}d_s {\rm Tr}(X_sQ_s),\;\; \widehat{h}_R(X) = \sum_{s\in \IrrQG}d_s {\rm Tr}(X_sQ^{-1}_s)$$
for $X = (X_s)_{s\in \IrrQG} \in c_{00}(\QDG)$, where $c_{00}(\QDG)$ is the subalgebra of $c_0(\widehat{\QG})$ consisting of finitely supported elements.

The latter algebra $c_{00}(\QDG)$ can be identified as a vector space with $\PolQG$ through the following Fourier transform on $\QDG$: 
	$$\F = \F^{\QDG}: c_{00}(\QDG) \to \PolQG,\;\; X \mapsto (X \cdot \widehat{h}_R \otimes id)U.$$
Here we used the notation $a\cdot\varphi:x\mapsto \varphi(xa)$, $\varphi\cdot a: x\mapsto \varphi(ax)$ for $\varphi$ a linear functional on some algebra $A$ and $a,x\in A$.
	
We can easily check that
	\begin{equation}\label{eq-Fourier}
	\F(e^s_{kl}) = d_s q_l(s)^{-1}u^{(s)}_{lk}
	\end{equation}
and so that $\F$ is a bijection. The algebraic dual $\PolQG^\dagger$ is naturally equipped with the structure of a multiplier Hopf $*$-algebra in the sense of Van Daele \cite{vdmult}, dual to the one of $\PolQG$. The Fourier transform $\F$ transfers the algebra structure of $\PolQG^\dagger$ to $\prod_{s\in \IrrQG} M_{n_s} = c_{00}(\QDG)^\dagger$, which is coming from the co-algebra structure of $\PolQG$. More precisely, for any $X,Y\in \prod_{s\in \IrrQG} M_{n_s}$ and $a\in \PolQG$ we have
	\begin{equation}\label{eq-alg-coalg}
	\la XY, \F^{-1}(a)\cdot \widehat{h}_R\ra = \la X\otimes Y,  (\F^{-1} \otimes \F^{-1})\circ \Delta (a)\cdot ( \widehat{h}_R \otimes  \widehat{h}_R) \ra.
	\end{equation}
Indeed, for $a= u^{(s)}_{kl}$ we can actually check that both of the above sides become $\sum^{n_s}_{j=1}(X_s)_{kj}(Y_s)_{jl}$ by recalling \eqref{eq-Fourier}. Note that the co-multiplication $\DcoP$ on $\ell^\infty(\widehat{\QG})$ can be uniquely extended to the following map (still denoted by $\DcoP$ by abuse of notation)
\begin{equation}\label{eq-Dual-comultiplication}
\DcoP : \prod_{s\in \IrrQG} M_{n_s} \to \prod_{s, t\in \IrrQG} M_{n_s} \otimes M_{n_t},
\end{equation}
which can be computed, up to unitary equivalence, using Proposition \ref{prop-fusion} to be explained later in the paper. See \cite[Section 1.6]{neshveyev+tuset} for other explanation.

The predual $\ell^1(\widehat{\QG})=A(\QG)$ of $\ell^\infty(\widehat{\QG}) = VN(\QG)$ is equipped with the algebra multiplication given by the preadjoint map $\DcoP_*$. We call $(A(\QG), \DcoP_*)$ the {\it Fourier algebra on $\QG$}. We have a natural embedding of $c_{00}(\QDG)$ into $A(\QG)$:
	\begin{equation}\label{eq-embed-A(G)}
	c_{00}(\QDG) \hookrightarrow A(\QG),\;\; X \mapsto X\cdot \widehat{h}_R.
	\end{equation}
Of course, there are many other possible embeddings of $c_{00}(\QDG)$ into $A(\QG)$ such as using left Haar weight, but we will stick to the above convention.

Using the above embedding we can extend the Fourier transform $\F$ to $A(\QG)$ as follows.
	$$\F : A(\Gb) \to C_r(\Gb),\;\; \psi \mapsto (\psi \otimes id)\mathbb{U},$$
where we use the same symbol by abuse of notation. Note that $\F$ is an injective homomorphism and the norm structure on $A(\QG)$ is given by the duality $(A(\QG), VN(\QG))$. Thus, for the element $X\cdot \widehat{h}_R \in A(\QG)$ with $X = (X_s)\in c_{00}(\widehat{\QG})$ we get the concrete norm formula as follows.
	$$||X\cdot \widehat{h}_R||_{A(\QG)} = \sum_{s\in \IrrQG} d_s \cdot ||X_sQ^{-1}_s||_{S^1_{n_s}},$$
where $S^1_m$ refers to the trace class operators acting on $\ell^2_m$. The Fourier transform $\F$ also transfers the $*$-structure on $\PolQG$ into $c_{00}(\QDG)\subseteq A(\QG)$. Indeed, for $X\in c_{00}(\QDG)$ we define $X^\star$ given by $\F(X^\star) = \F(X)^*$. Then from the fact that $(\widehat{S}\otimes id)\mathbb{U}^* = \mathbb{U}$ we can easily check that
	\begin{equation}\label{eq-inv}
	\F(X^\star) = \F((\widehat{h}_R\cdot X^*)\circ \widehat{S}),
	\end{equation}
where $\widehat{S}$ is the antipode of the dual quantum group $\QDG$.
	
In this article we are often interested in a homomorphism
	$$\varphi: {\rm Pol}(\QG) \to B(H)$$
for some Hilbert space $H$ and its transferred version
	$$\tilde{\varphi}:=\varphi \circ \F: c_{00}(\widehat{\QG}) \to B(H).$$
The map $\tilde{\varphi}$ can be understood as an element of $v = v_\varphi\in \prod_{s\in \IrrQG} (M_{n_s}\otimes B(H))$, which is given by
	\begin{equation}\label{element-v}
	v_\varphi = (id \otimes \varphi)\mathbb{U}
	\end{equation}
or equivalently
	$$v_\varphi = (v^{(s)})_{s\in \IrrQG}\;\; \text{with}\;\; v^{(s)} = (id\otimes \varphi) (u^{(s)}).$$
Indeed, for $X \in c_{00}(\QDG)$ we have
	$$\la \tilde{\varphi}, X\ra = \varphi \circ \F (X) = \varphi((X\cdot \widehat{h}_R \otimes id)\mathbb{U}) = (X\cdot \widehat{h}_R \otimes \varphi)\mathbb{U} = \la X\cdot \widehat{h}_R, v_\varphi \ra.$$
From this identification we can easily observe that $\tilde{\varphi}$ extends to a completely bounded map from $A(\QG)$ to $B(H)$ if and only if $v_\varphi \in VN(\QG) \bar{\otimes} B(H)$, the spatial tensor product of von Neumann algebras $VN(\QG)$ and $B(H)$ with
	\begin{equation}\label{eq-OS-knowledge}||\tilde{\varphi}||_{cb} = ||v_\varphi||_{VN(\QG) \bar{\otimes} B(H)} =  \sup_{s\in \IrrQG} ||v^{(s)}||_{M_{n_s}\otimes B(H)}<\infty,
	\end{equation}
which we will frequently use later in the weighted context. Here, we are assuming that the readers are familiar with standard concepts of operator spaces such as completely bounded maps and the cb-norm denoted by $||\cdot||_{cb}$. See \cite{Pisier} for the details of operator space theory. For example, we used the fact that $CB(M_*, N) \cong M \bar{\otimes} N$ completely isometrically for von Neumann algebras $M$ and $N$ for the above \eqref{eq-OS-knowledge}.

\subsection{Closed quantum subgroups}
A compact quantum group $\QH$ is a closed quantum subgroup of another compact quantum group $\QG$ (see \cite{Tomatsu} for the details) if there is a surjective $*$-homomorphism
	$$R_\QH : {\rm Pol}(\QG) \to {\rm Pol}(\QH)$$
such that
	$$\Delta_\QH \circ R_\QH = (R_\QH \otimes R_\QH)\circ \Delta_\QG,$$
where $\Delta_\QH$ and $\Delta_\QG$ are the corresponding co-multiplications. In this case there is a uniquely determined normal unital $*$-homomorphism
	\begin{equation}\label{eq-embedding}
	\iota: VN(\QH) \to VN(\QG)
	\end{equation}
such that
	$$\DcoP_{\QDG} \circ \iota = (\iota \otimes \iota)\circ \DcoP_{\QDH}.$$	
The relationship between $R_\QH$ and $\iota$ is given by the equation
	\begin{equation}\label{eq-restriction-embedding}
	(id\otimes R_\QH) \mathbb{U}_\QG = (\iota \otimes id)\mathbb{U}_\QH,
	\end{equation}
where $\mathbb{U}_\QG$ and $\mathbb{U}_\QH$ are multiplicative unitaries of $\QG$ and $\QH$, respectively (\cite[Lemma 2.10]{Tomatsu}).

\subsection{Quantum double, $SL_q(2,\Comp)$ and the Iwasawa decomposition}\label{subsec-SLq(2)}

For a compact quantum group $\QG$ we construct the quantum double $\QG \Join \QDG$ as follows. The associated $C^*$-algebra is given by $C_0(\QG \Join \QDG) := C(\QG) \otimes c_0(\QDG)$ with the co-multiplication
	$$\Delta_\Comp = (id \otimes \Sigma_{\mathbb{U}} \otimes id)(\Delta \otimes \DcoP),$$
where $\Sigma_{\mathbb{U}}$ is the $*$-isomorphism given by 
	$$\Sigma_{\mathbb{U}} : C(\QG) \otimes c_0(\QDG) \to c_0(\QDG) \otimes C(\QG),\;\; a\otimes x \mapsto \mathbb{U}(x\otimes a)\mathbb{U}^*.$$ The left (and right) Haar weight on $C_0(\QG \Join \QDG)$ is given by $h\otimes \widehat{h}_{R}$. We denote the quantum double $\QG \Join \QDG$ shortly by $\QG_\Comp$ when $\QG = G_q$, the Drinfeld-Jimbo $q$-deformation of a compact, semi-simple and simply connected Lie group $G$. We are interested in the $C^*$-algebra spectrum ${\rm sp}\, C_0(\QG_\Comp)$, i.e., the equivalence classes of irreducible $*$-representations of the $C^*$-algebra $C_0(\QG_\Comp)$. We know that
		\begin{equation}\label{eq-spectrum}
		{\rm sp}\,C_0(\QG_\Comp) = {\rm sp}\,(C(\QG) \otimes c_0(\QDG)) \cong {\rm sp}\,C(\QG) \times {\rm sp}\,c_0(\QDG)
		\end{equation}
as topological spaces via the correspondence
	$$\pi \in {\rm sp}\,C_0(\QG_\Comp) \mapsto (\pi_c, \pi_d) \in {\rm sp}\,C(\QG) \times {\rm sp}\,c_0(\QDG)$$
since $c_0(\QDG)$ is a direct sum of matrix algebras. We remark that the above $C^*$-algebra spectra are known, namely ${\rm sp}\,C(\QG)$ (\cite{Soi}) and ${\rm sp}\,c_0(\QDG) \cong {\rm Irr}(\QG)$ (\cite{Kus}). It will turn out that the former information is not needed for us, but the latter information will be crucial. Moreover, the latter identification has been concretely established in \cite{podles+woronowicz90} in the case of $\QG = SU_q(2)$, $-1<q\ne 0<1$, which we describe below.

The $C^*$-algebra $C(SU_q(2))$ is the universal $C^*$-algebra with the two generators $a_q$ and $c_q$ satisfying the commutation relations which make $\left(\begin{array}{ll} a_q & -qc_q^* \\ c_q & a_q^* \end{array}\right)$ unitary, i.e.,
	$$a^*_qa_q+c^*_qc_q = I = a_qa^*_q+q^2c^*_qc_q,\;\, c^*_qc_q=c_qc^*_q,\;\, a_qc_q=qc_qa_q,\;\, a_qc^*_q=qc^*_qa_q.$$
Moreover, the co-multiplication is determined by
	$$\Delta(a_q) = a_q\otimes a_q - q c_q^* \otimes c_q,\;\; \Delta(c_q) = c_q\otimes a_q + a^*_q \otimes c_q.$$
The representation theory of $SU_q(2)$ is well-known. We have
	$${\rm Irr}(SU_q(2)) = \frac{1}{2}\z_+\;\; \text{with}\; {\rm dim}\, s = 2s+1\;\;\text{for}\;\; s\in {\rm Irr}(SU_q(2)).$$ Moreover, we have the following fusion rule:
	$$u^{(1/2)} \tp u^{(s)} \cong u^{(s+1/2)} \oplus u^{(s-1/2)},\; s\in \frac{1}{2}\z_+,\, s\ge 1/2$$
with $u^{(0)} = 1$ being the trivial representation. See \cite{woronowicz87} for more details.

In this paper we will focus on the quantum double $SU_q(2) \Join \widehat{SU_q(2)}$, which is called the quantum Lorentz group $SL_q(2,\Comp)$ since it is regarded as {\it the} complexification $SU_q(2)_\Comp$ of $SU_q(2)$.

\section{The Beurling-Fourier algebra of a compact quantum group}\label{sec-BF-alg-Def}

We would like to consider a weighted version of $A(\mathbb{G})$ following the construction for compact groups, which begins with the following understanding of the coproduct on $VN(\QG)$.

\begin{prop}\label{prop-fusion}
Let $\Gb$ be a compact quantum group and $X = (X_s)_{s\in\IrrQG} \in VN(\QG)$. Then for each $t_1, t_2 \in \IrrQG$ we have the following unitary equivalence
$$\DcoP(X)(t_1, t_2) \cong \bigoplus_{u^{(s)} \subseteq u^{(t_1)} \tp u^{(t_2)}} X_s.$$
Moreover, the unitary matrix for the above equivalence is the same as the one in the decomposition of $u^{(t_1)} \tp u^{(t_2)}$.
\end{prop}

\begin{proof}
See \cite[Proposition 4.2]{KS17}, which corrects a mistake regarding multiplicity in \cite[Proposition 4.3]{FLS}.
\end{proof}	
	
The above Proposition enables us to proceed just as in the compact group case (\cite{LS, LST}), so that we define the (central) weights as follows.

\begin{defn}
We call $w : \IrrQG \to (0,\infty)$ a (central) weight on the dual of $\mathbb{G}$ if it satisfies
\begin{equation}\label{def-weight}
w(s) \le w(t) w(t')
\end{equation}
whenever $u^{(s)} \subseteq u^{(t)} \tp u^{(t')}$. We assume in addition the weight $w$ to be bounded below, i.e. there is $\delta>0$ such that $w(s) \ge \delta$ for all $s\in \IrrQG$.
\end{defn}

The weights give rise to weighted versions $A(\QG, w)$ and $VN(\QG, w^{-1})$ of the spaces $A(\QG)$ and $VN(\QG)$. We begin with $VN(\QG, w^{-1})$ given by
	$$VN(\QG, w^{-1}) := \{ Y = (Y_s)_{s\in \IrrQG} \in \prod_{s\in \IrrQG} M_{n_s}: \sup_{s\in \IrrQG} \frac{||Y_s||_{M_{n_s}}}{w(s)} <\infty \},$$
where $||Y||_{VN(\QG, w^{-1})} := \sup_{s\in \IrrQG} w(s)^{-1} \cdot ||Y_s||_{M_{n_s}}$.
We let $W = (w(s)I_s)_{s\in \IrrQG} \in \prod_{s\in \IrrQG} M_{n_s}$ be the operator associated to the weight $w$ and define the natural isometry
	$$\Phi: VN(\QG) \to VN(\QG, w^{-1}),\;\; X = (X_s) \mapsto XW = (w(s)X_s).$$
We now define the weighted space $A(\QG, w)$ to be the unique predual of $VN(\QG, w^{-1})$. Similar to the embedding \eqref{eq-embed-A(G)} we consider another embedding
	\begin{equation}\label{eq-embed-A(G,w)}
	c_{00}(\QDG) \hookrightarrow A(\QG, w),\;\; X \mapsto X\cdot \widehat{h}_R.
	\end{equation}
Then we have for any $X = (X_s) \in c_{00}(\QDG)$ that
	\begin{align*}
	||X\cdot \widehat{h}_R||_{A(\QG, w)}
	& = \sup_{\|Y\|_{VN(\QG, w^{-1})}\le 1}|\la X\cdot \widehat{h}_R, Y \ra|\\
	& = \sup_{\|Y\|_{VN(\QG, w^{-1})}\le 1}|\widehat{h}_R(YX)|\\
	& = \sum_{s\in \IrrQG} d_s w(s) \cdot ||X_sQ^{-1}_s||_{S^1_{n_s}}.
	\end{align*}
By density we can see that the above formula extends to general elements in $A(\QG, w)$, so that we have
	$$A(\QG, w) = \{ X\cdot \widehat{h}_R: X = (X_s) \in \prod_{s\in \IrrQG} M_{n_s},\;\sum_{s\in \IrrQG} d_s w(s) \cdot ||X_sQ^{-1}_s||_{S^1_{n_s}} <\infty \}.$$
Note that the preadjoint $\Phi_*$ of $\Phi$ becomes an isometry between $A(\QG,w)$ and $A(\QG)$ as follows.
	$$\Phi_*: A(\QG,w) \to A(\QG),\; X\cdot \widehat{h}_R \mapsto Y\cdot \widehat{h}_R,$$
where $Y_s = w(s) X_s$, $s\in \IrrQG$. The condition $w$ being bounded below gives us the inclusion $A(\QG, w) \subseteq A(\QG)$. We equip $VN(\QG, w^{-1})$ with the operator space structure making the above map $\Phi$ a complete isometry, so that the predual $A(\QG,w) = VN(\QG, w^{-1})_*$ also has a natural operator space structure. Since $w\times w$ is clearly a weight on the dual of $\QG \times \QG$, we can consider the associated weighted spaces $VN(\QG\times \QG, w^{-1}\times w^{-1})$ and $A(\QG \times \QG, w\times w)$. Then, we clearly have
	$$VN(\QG\times \QG, w^{-1}\times w^{-1}) \cong (A(\QG,w) \prt A(\QG,w))^*$$
completely isometrically via the formal identity. Here, $\prt$ is the projective tensor product of operator spaces.

Now Proposition \ref{prop-fusion} implies that \eqref{def-weight} is equivalent to
	$$||\DcoP(W)(W^{-1} \otimes W^{-1})||_\infty \le 1,$$
where we use the extended co-multiplication $\DcoP$ from \eqref{eq-Dual-comultiplication}, so that we have a normal complete contraction
	$$\tilde{\DcoP}: VN(\QG) \to VN(\QG \times \QG),\;\; X \mapsto \DcoP(X)\DcoP(W)(W^{-1} \otimes W^{-1}).$$
Equivalently, we have a weak$^*$-weak$^*$ continuous complete contraction
	$$\DcoP: VN(\QG, w^{-1}) \to VN(\QG \times \QG, w^{-1}\times w^{-1}),\; Y\mapsto \DcoP(Y).$$
Thus, for a weight $w: \IrrQG \to (0,\infty)$ on the dual of $\QG$ the space $(A(\QG, w), \DcoP_*)$ becomes a completely contractive Banach algebra.

\begin{defn}\label{def-weights}
For a weight $w$ on the dual of $\mathbb{G}$ we call $(A(\QG, w), \DcoP_*)$ a (central) Beurling-Fourier algebra on $\mathbb{G}$.
\end{defn}

The following are examples of weights.

\begin{ex} \label{ex-central-weights}
	Let $\QG$ be a quantum group such that $\IrrQG$ is finitely generated and let $\tau$ be the word length function $\tau$ associated to a fixed finite generating set $S\subseteq \IrrQG$, i.e. for $\pi \in \IrrQG$ we have
		$$\tau(\pi) = \min\{N: \pi \subseteq \sigma_1 \tp \cdots \tp \sigma_N,\; \sigma_j \in S\}.$$
	\begin{enumerate}
		
		\item For $\alpha \ge 0$ we define the polynomially growing weights $w_\alpha$ given by
			$$w_\alpha(s) := (1+\tau(s))^\alpha,\; s\in \IrrQG.$$
			
		\item For $\beta \ge 1$ we define the exponentially growing weights $w_\beta$ given by
			$$w_\beta(s) := \beta^{\tau(s)},\; s\in \IrrQG.$$
	\end{enumerate}
\end{ex}

The weights defined are bounded below by one, so we have $A(\mathbb{G},w)\subseteq A(\mathbb{G})$ contractively. Now using the embedding \eqref{eq-embed-A(G,w)} and the same argument as \eqref{eq-OS-knowledge} we get the following.

\begin{prop}\label{prop-cb-norm-weighted}
Let $\varphi: {\rm Pol}(\QG) \to B(H)$ be a homomorphism for some Hilbert space $H$ and $v_\varphi = (v^{(s)})_{s\in \IrrQG} \in \prod_{s\in \IrrQG} (M_{n_s}\otimes B(H))$ be the associated element as in \eqref{element-v}. Then the transferred homomorphism $\tilde{\varphi}=\varphi \circ \F: c_{00}(\widehat{\QG}) \to B(H)$ extends to a completely bounded map $\tilde{\varphi}: A(\QG, w) \to B(H)$ if and only if $v_\varphi = (v^{(s)})_s \in \prod_{s\in \IrrQG} (M_{n_s}\otimes B(H))$ satisfies the following norm condition:
$$||\tilde{\varphi}||_{cb} = \sup_{s\in \IrrQG} w(s)^{-1}||v^{(s)}||_{M_{n_s}\otimes B(H)}<\infty.$$
\end{prop}

\section{The Gelfand spectrum of Beurling-Fourier algebra and complexification}\label{sec-characters}
	
	We begin with the first scenario, namely investigating the Gelfand spectrum of the Beurling-Fourier algebra $A(\QG,w)$. We expect a connection with some complexification of $\QG$, which we describe below.
	
	\begin{defn}
	We define $\widetilde{\mathbb{G}}$ to be the Gelfand spectrum ${\rm Spec}\, A(\mathbb{G})$ of the Banach algebra $A(\mathbb{G})$, i.e. the space of all non-zero homomorphisms from the Banach algebra $A(\mathbb{G})$ to $\mathbb{C}$, and $\widetilde{\mathbb{G}}_\mathbb{C}$ to be ${\rm Spec}\,{\rm Pol}(\mathbb{G})$, the space of all non-zero homomorphisms from the algebra ${\rm Pol}(\mathbb{G})$ to $\mathbb{C}$. We define the group multiplications on $\widetilde{\mathbb{G}}$ and $\widetilde{\mathbb{G}}_\Comp$ by the multiplications of $VN(\QG)$ and $\prod_{s\in \IrrQG} M_{n_s}$, respectively. The groups $\widetilde{\mathbb{G}}$ and $\widetilde{\mathbb{G}}_\Comp$ can be regarded as topological groups with the weak$^*$-topologies on $VN(\QG)$ and $\prod_{s\in \IrrQG} M_{n_s}$, respectively.
	\end{defn}

\begin{rem}\label{rem-abstract-Lie}
	\begin{enumerate}
	
		\item In \cite[Chapter 3]{neshveyev+tuset} the groups $\widetilde{\mathbb{G}}$ and $\widetilde{\mathbb{G}}_\mathbb{C}$ are denoted by $H^1(\QDG, \tor)$ and $H^1(\QDG, \Comp^*)$, respectively.  For a compact connected Lie group $\QG = G$ we actually know that $\widetilde{\mathbb{G}}_\mathbb{C}$ is the Lie group complexification of $G$. See \cite[III. 8]{BtD} or \cite[Theorem 3.2.2]{neshveyev+tuset} for example.
		
		\item In other words, we can write
			$$\widetilde{\QG} = \{ X \in VN(\QG): \DcoP(X) = X\otimes X, \; X\ne 0\}$$
		and
			$$\widetilde{\mathbb{G}}_\mathbb{C} = \{ X \in \prod_{s\in \IrrQG}M_{n_s}: \DcoP(X) = X\otimes X, \; X\ne 0\},$$
		where $\DcoP$ is the extended co-multiplication from \eqref{eq-Dual-comultiplication}. An observation right after \cite[Definition 3.1.3]{neshveyev+tuset} actually tells us that we may replace the condition $X\ne 0$ by invertibility in the above, i.e. we have
			$$\widetilde{\mathbb{G}}_\mathbb{C} = \{ X \in \prod_{s\in \IrrQG}M_{n_s}: \DcoP(X) = X\otimes X, \; X\; \text{invertible}\}.$$
		For the case of $\widetilde{\QG}$ we may even use unitarity, i.e. we have
			$$\widetilde{\QG} = \{ X \in VN(\QG): \DcoP(X) = X\otimes X, \; X\; \text{unitary}\}.$$
		
		\item The group $\widetilde{\QG}$ has been already defined and investigated in \cite{KalNeu} for a general locally compact quantum group $\QG$. The definition of $\widetilde{\QG}$ in this paper, however, is given in an equivalent form. See Proposition \ref{prop-max-classical-subgroup} below for more properties of $\widetilde{\QG}$.
		
		\item In \cite{McK} K. McKennon developed an {\it abstract Lie theory} for general, possibly non-Lie, compact groups. Recall that if $G$ is a compact group, then the Tannaka-Krein duality (\cite[(30.5)]{HeRo}) tells us that
	$G \cong \{X \in VN(G): \DcoP_G (X) = X\otimes X,\; X\ne 0\}.$
Motivated from this McKennon defined $G_\Comp$, $\mathfrak{g}_\Comp$ and $\mathfrak{g}$ as follows.
	\begin{align*}
		G_\Comp &:= \{X \in \prod_{s\in {\rm Irr}(G)}M_{n_s}: \DcoP_G (X) = X\otimes X,\; X\ne 0\},\\
		\mathfrak{g}_\Comp &:= \{ x\in \prod_{s\in {\rm Irr}(G)}M_{n_s}: \DcoP_G (x) = x\otimes I + I \otimes x,\; x\ne 0\},\\
		\mathfrak{g} &:= \{ x\in \mathfrak{g}_\Comp: x = -x^*\}
	\end{align*} 
It has been shown (\cite[Proposition 3]{CMc}) that $x\in \mathfrak{g}_\Comp$ (resp. $\mathfrak{g}$) if and only if $\exp(tx) \in G_\Comp$ (resp. $G$) for all $t\in \Real$, where $\exp: \mathfrak{g}_\Comp \to G_\Comp \subseteq \prod_{s\in {\rm Irr}(G)}M_{n_s}$ is the exponential map for $G_\Comp$, which actually becomes matrix exponential map when we go down to each component level, namely, $M_{n_s}$. Moreover, it also has been shown (\cite[Proposition 3, Corollary 1]{CMc}) that for any continuous homomorphism $\varphi: \Real \to G_\Comp$ (resp. $\psi: \Real \to G$) there is $y\in \mathfrak{g}_\Comp$ (resp. $\mathfrak{g}$) such that $\varphi(t) = \exp(ty)$. Finally, we have a Cartan type decomposition (\cite[Corollary 1, Theorem 3]{McK})
	\begin{equation}\label{eq-Cartan}
	G_\Comp \cong G \cdot \exp(i\mathfrak{g}),
	\end{equation}
	where the corresponding identification map is a homeomorphic group isomorphism.
		
		\item Note that for a noncommutative C$^*$-algebra $A$, its Gelfand spectrum $${\rm Spec}\,A=\{\mbox{characters on }A\}$$ is only a subset of its C$^*$-algebraic spectrum $${\rm sp}\,A=\{\text{irred.\ $*$-representations of $A$ modulo unitary equivalence}\}.$$

	\end{enumerate}	
\end{rem}

The group $\widetilde{\QG}$ can be understood as a maximal classical closed quantum subgroup of $\QG$.

\begin{prop}\label{prop-max-classical-subgroup}
The topological group $\widetilde{\QG}$ is a compact group. Moreover, it is a maximal classical closed quantum subgroup of $\QG$ in the sense that for any classical closed quantum subgroup $H$ of $\QG$ we have an injective continuous group homomorphism from $H$ into $\widetilde{\QG}$.
\end{prop}
\begin{proof}
We recall that the group $\widetilde{\QG}$ is equipped with the weak$^*$-topology inherited from $VN(\QG) = A(\QG)^*$.

We first check the compactness of $\widetilde{\QG}$. Let $(X_i)_{i\in \I}$ be a net in $\widetilde{\QG}$. Since each $X_i$ is a unitary, the Alaoglu theorem gives us a subnet $(X_j)$ of $(X_i)_{i\in \I}$ such that $X_j \to X\in VN(\QG)$ in the weak$^*$-topology. It is clear that we have $\DcoP(X) = X\otimes X$. Thus, it is enough to check that $X \ne 0$ to see that $X\in \widetilde{\QG}$. Indeed, for the unit $1_0$ of the algebra $A(\QG)$ we have
	$$\la Y, f \ra = \la Y, f * 1_0 \ra = \la Y, f \ra \cdot \la Y, 1_0 \ra$$
for any $Y\in \widetilde{\QG}$ and $f\in A(\QG)$, which clearly implies that $\la Y, 1_0 \ra = 1$ and consequently $\la X,1_0\ra =1$ and $X\ne 0$.

Secondly, we show that $\widetilde{\QG}$ is actually a quantum closed subgroup of $\QG$. Indeed, it is enough to consider the following $*$-homomorphism.
	\begin{equation}\label{eq-restriction-max-classical}
	\pi: C_u(\QG) \to C(\widetilde{\QG}),\; a\mapsto \pi(a)
	\end{equation}
with $\pi(a)(Y) = \theta_Y(a)$, where $\theta_Y: C_u(\QG) \to \Comp$ is the $*$-homomorphism extended from the $*$-homomorphism $\Theta:{\rm Pol}(\mathbb{G})\cong c_{00}(\QDG) \to \Comp,\; X\mapsto \la X\cdot \widehat{h}_R, Y\ra$ by the universality of $C_u(\QG)$.
Indeed, we have $\Theta(X^\star) = \la (\widehat{h}_R\cdot X^*)\circ \widehat{S}, Y\ra = \widehat{h}_R(X^*\widehat{S}(Y)) = \widehat{h}_R(X^*Y^*) = \overline{\la X\cdot \widehat{h}_R, Y\ra}$, where we used the fact that $\widehat{S}(Y) = Y^*$ since $\DcoP(Y) = Y\otimes Y$. For the surjectivity of $\pi$ we consider the unital $*$-subalgebra $\pi({\rm Pol}(\QG))$ of $C(\widetilde{\QG})$. The weak$*$-density of $c_{00}(\QDG)$ in $VN(\QG)$ tells us that the algebra $\pi({\rm Pol}(\QG))$ separates points in $\widetilde{\QG}$, which gives us the density of $\pi({\rm Pol}(\QG))$ in $C(\widetilde{\QG})$ by the Stone-Weierstrass theorem. Since the images of $*$-homomorphisms on $C^*$-algebras are always closed, we get the surjectivity of $\pi$. The equality $(\pi\otimes \pi)\circ \Delta^u_{\QG} = \Delta_{\widetilde{\QG}} \circ \pi$ on $\PolQG$ is direct from \eqref{eq-alg-coalg} and we can apply density and continuity for the whole algebra $C_u(\QG)$.

For the last statement we consider a classical closed quantum subgroup $H$ of $\QG$. Then we have the injective $*$-homomorphism $\iota: VN(H) \to VN(\QG)$. For $h\in H$ we have $\lambda_H(h) \in VN(H)$ satisfying $\DcoP_H(\lambda_H(h)) = \lambda_H(h) \otimes \lambda_H(h)$. Thus, we have $\iota(\lambda_H(h))$ is a group-like element in $VN(\QG)$. Finally, we have $H \to \widetilde{\QG},\; h\mapsto \iota(\lambda_H(h))$ is the continuous injective homomorphism we wanted.
\end{proof}

\begin{rem}
Note that Proposition \ref{prop-max-classical-subgroup} is already known by \cite[Remark 4.3 (2)]{DKSS}. We included the proof for the sake of simpler arguments available for compact cases. Note also that the paper \cite{McK} deals with the group $\widetilde{\QG}_\Comp$ only for the case of classical compact groups.
\end{rem}

Now we need to compare the group $\widetilde{\QG}_\Comp$ with $(\widetilde{\mathbb{G}})_\mathbb{C}$, the abstract complexification of $\widetilde{\mathbb{G}}$.

	\begin{thm}\label{thm-max-group}
	The group $\widetilde{\mathbb{G}}_\mathbb{C}$ is isomorphic to the abstract complexification $(\widetilde{\mathbb{G}})_\mathbb{C}$ of $\widetilde{\mathbb{G}}$ as topological groups.
	\end{thm}
\begin{proof}
Recall the map $\pi$ from \eqref{eq-restriction-max-classical} characterizing closed quantum subgroup structure of $\widetilde{\QG}$. Then it is straightforward to check that $(id\otimes \pi)\mathbb{U}_\QG$ is a vector-valued function on $\widetilde{\QG}$ given by
	$$(id\otimes \pi)\mathbb{U}_\QG(V) = \oplus_{s\in \IrrQG} V_s$$
for $V = (V_s)_{s\in \IrrQG} \in \widetilde{\QG} \subseteq VN(\QG)$. Now, the equation \eqref{eq-restriction-embedding} tells us that the associated embedding $\iota: VN(\widetilde{\QG}) \to VN(\QG)$ (as in \eqref{eq-embedding}) is actually given by $\iota(\lambda_{\widetilde{\QG}}(V)) = V$ for $V = (V_s)_{s\in \IrrQG}\in \widetilde{\QG}$.

We will now establish a Cartan type decomposition for $\widetilde{\mathbb{G}}_\mathbb{C}$ using a standard argument. We pick any $X \in \widetilde{\QG}_\Comp$, then the polar decomposition $X = U |X|$ also satisfies that $U, |X| \in \widetilde{\QG}_\Comp$. Indeed, we know that $X^*X \in \widetilde{\QG}_\Comp$, so that $|X| \in \widetilde{\QG}_\Comp$ via functional calculus. Since $X$, and consequently $|X|$, are invertible ((2) of Remark \ref{rem-abstract-Lie}) we also get $U\in \widetilde{\QG}_\Comp$.
Since $U$ is unitary we actually have $U \in \widetilde{\QG}$. Moreover, for any $z\in \Comp$ we have $\DcoP_\QG |X|^z = |X|^z \otimes |X|^z$ by functional calculus. In particular, $\DcoP_\QG |X|^{it} = |X|^{it} \otimes |X|^{it}$ for all $t\in \Real$,
which means that we have a continuous homomorphism from $\Real$ to $\widetilde{\QG}$, namely, $t\mapsto |X|^{it}$.
By (4) of Remark \ref{rem-abstract-Lie} we have an element $y$ in the abstract Lie algebra $\widetilde{\mathfrak{g}}$ of $\widetilde{\QG}$ such that $|X|^{it} = \exp_{\widetilde{\QG}}(-ty)$, $t\in \Real$ as elements in $\widetilde{\QG}\subseteq VN(\widetilde{\QG})$, where $\exp_{\widetilde{\QG}}$ is the exponential map from $\widetilde{\mathfrak{g}}$ to $\widetilde{\QG}\subseteq VN(\widetilde{\QG})$.
From the observation we had in the begining of this proof we actually get
	$$|X|^{it} = \iota(\lambda_{\widetilde{\QG}}(\exp_{\widetilde{\QG}}(-ty)),\; t\in \Real$$ as elements in $VN(\QG)$. By the uniqueness of analytic functions we have $|X|^z = \iota(\lambda_{\widetilde{\QG}}(\exp_{\widetilde{\QG}}(izy))$ for any $z\in \Comp$, so that we get $|X| = \iota(\lambda_{\widetilde{\QG}}(\exp_{\widetilde{\QG}}(iy))$ for $z=1$. Consequently, we have
	$$X = U|X| = U \cdot  \iota(\lambda_{\widetilde{\QG}}(\exp_{\widetilde{\QG}}(iy))$$
with $U \in \widetilde{\QG}_\Comp$ and $y\in \widetilde{\mathfrak{g}}$. From the invertibility of $X$ we know that $U$ and $y$ are uniquely determined by $X$.
Moreover, the maps $X\mapsto |X|$ and $X\mapsto U = X|X|^{-1}$ are clearly continuous. Since the map $\widetilde{\mathfrak{g}} \to (\widetilde{\QG})_\Comp,\; y\mapsto \exp_{\widetilde{\QG}}(iy)$ is a homeomorphism onto its image by \eqref{eq-Cartan}, we actually get a homeomorphism
	$$\widetilde{\QG}_\Comp \cong \widetilde{\QG} \times \widetilde{\mathfrak{g}},\; X \mapsto (U, y).$$
Combining with the homeomorphism
	$$(\widetilde{\QG})_\Comp \cong \widetilde{\QG} \times \widetilde{\mathfrak{g}},\; \iota(\lambda_{\widetilde{\QG}}(U)) \exp_{\widetilde{\QG}}(iy) \mapsto (U, y)$$
from \eqref{eq-Cartan} we can conclude that the map $X \in \widetilde{\QG}_\Comp \mapsto \iota(\lambda_{\widetilde{\QG}}(U)) \exp_{\widetilde{\QG}}(iy) \in (\widetilde{\QG})_\Comp$ is a homeomorphic group isomomorphism.
\end{proof}

\begin{rem}
From the definition we have ${\rm Spec}A(\QG, w)\subseteq {\rm Spec}\,{\rm Pol}(\mathbb{G}) =  \widetilde{\QG}_\Comp$. The above Theorem \ref{thm-max-group} allows us to use Beurling-Fourier algebras to {\it detect} the structure of the complexification of the maximal classical closed subgroup of $\QG$, which will be demonstrated in the following subsections.
\end{rem}

In the following subsections we will consider several concrete examples.

\subsection{The case of free orthogonal quantum group $O_F^+$}

We recall the definition of the universal (or free) orthogonal compact quantum group $O_F^+$ from \cite{VanWa} for $F\in GL_n(\mathbb{C})$ with $F\bar{F}=\pm I$. The $C^*$-algebra $C(O_F^+)$ is the universal $C^*$-algebra with the generators $U = (u_{ij})^n_{i,j=1}$ satisfying the relations 
	\begin{equation}\label{eq-rel-orth}
	\begin{cases}UU^* = U^*U = I,\\ UF = F\overline{U},\end{cases}
	\end{equation}
where $\overline{U} = (u^*_{ij})^n_{i,j=1}$, so that we clearly have $C(O_F^+) = C_u(O_F^+)$. The co-multiplication on $C(O_F^+)$ is determined by $\Delta(u_{ij}) = \sum^n_{k=1}u_{ik}\otimes u_{kj}$, $1\le i,j\le n$. The fusion rule for $O_F^+$ is the same as the one of $SU(2)$ (\cite{ban-onplus}), which means that ${\rm Irr}(O_F^+) = \frac{1}{2}\z_+$, $u^{(0)} = 1$, $u^{(1/2)} = U = (u_{ij})^n_{i,j=1}$ and
	$$u^{(1/2)} \tp u^{(s)} \cong u^{(s+1/2)} \oplus u^{(s-1/2)} \quad\mbox{ for } s\ge \frac{1}{2}.$$
Note that ${\rm Irr}(O_F^+)$ is singly generated by $u^{(1/2)}$ and the above fusion rule tells us that the corresponding word length function $\tau$ is given by
	$$\tau(s) = 2s,\; s\in \frac{1}{2}\z_+.$$
The classical dimension is given by $n_s=U_{2s}(n)$, where $(U_{2s})_{n\ge 0}$ denotes the Chebyshev polynomials of the second kind.

\begin{thm}\label{thm-free-ortho-spec}
Let $\beta\ge 1$.
We have $\widetilde{(O_F^+)}_\Comp \cong \{V\in GL_n(\Comp): VFV^t = F\}$ as topological groups and for the exponential weight $w_\beta$ on the dual of  $C(O_F^+)$ as in Example \ref{ex-central-weights} we have
	$${\rm Spec} A(O_F^+, w_\beta) \cong \{ V\in GL_n(\Comp) : VFV^t = F,\;\; 1\le ||V||_\infty \le \beta\}$$
as topological spaces.
\end{thm}
\begin{proof}
A non-zero character $\varphi : {\rm Pol}(O_F^+) \to \mathbb{C}$ is determined by two matrices $V = (\varphi(u_{ij}))^n_{i,j=1}$ and $W = (\varphi(u^*_{ji}))^n_{i,j=1}$ in $M_n$. It is easy to see that \eqref{eq-rel-orth} implies that $V$ and $W$ satisfy
\begin{equation}\label{eq-rel-orth2}
\begin{cases}VW = WV = I,\\ VF = FW^t\end{cases} \Leftrightarrow \begin{cases}VW = WV = I,\\ VFV^t = F,\end{cases}
\end{equation}
where $W^t$ refers to the transpose of $W$. Conversely, any matrices $V$ and $W$ satisfying \eqref{eq-rel-orth2} define a non-zero character $\varphi : {\rm Pol}(O_F^+) \to \Comp$. Since $W=V^{-1}$ this gives the first assertion.
Moreover, in this case we know that $I_n \otimes \varphi(u^{1/2}) = V$. Together with the above fusion rule we can easily see that the map $\varphi \mapsto V$ is a homeomorphism between ${\rm Spec} {\rm Pol}(O_F^+)$ and $\{V\in GL_n(\Comp): VFV^t = F\}$ with the subspace topology inherited from the usual topology of $GL_n(\mathbb{C})$.

For the second statement we note that the fusion rule is preserved by the character $\varphi$, so that for any $s\in \frac{1}{2}\z_+$ with $s\ge 1/2$ we get
	$$v^{(1/2)} \otimes v^{(s)} \cong v^{(s+1/2)} \oplus v^{(s-1/2)},$$
where $v\in \prod_{s\in \frac{1}{2}\z_+}M_{n_s}$ is the element corresponding to $\varphi$. Thus, we have the following recurrence relation for the sequence of the norms $||v^{(s)}||_\infty$.
	$$||v^{(1/2)}||_\infty \cdot ||v^{(s)}||_\infty = \max\{ ||v^{(s+1/2)}||_\infty, ||v^{(s-1/2)}||_\infty\}.$$
From \eqref{eq-rel-orth2} we have $||F||_\infty \le ||V||_\infty \cdot ||F||_\infty \cdot ||V^t||_\infty$, which means that
	$$||v^{(1/2)}||_\infty = ||V||_\infty \ge 1.$$
Thus, a simple induction with the observation that $v^{(0)} = \varphi(1) = 1$ tells us
	$$||v^{(s)}||_\infty = ||V||^{2s}_\infty,\;\; s\in \frac{1}{2}\z_+,$$
which gives us the bijection
	$$\varphi \in {\rm Spec} A(O_F^+, w_\beta) \mapsto V\in GL_n(\Comp)\; \text{such that}\; VFV^t = F,\, 1\le ||V||_\infty \le \beta.$$
Continuity of the above map is clear from the 	continuity of the map $\varphi \in {\rm Spec} {\rm Pol}(O_F^+) \mapsto V\in GL_n(\Comp)\; \text{such that}\; VFV^t = F.$ Now we recall that a continous bijection from a compact space into Hausdorff space is a homeomorphism, which gives us the conclusion we wanted.

\end{proof}

\begin{rem}
As in the classical case we can see that
	$$\bigcup_{\beta\ge 1}{\rm Spec}\, A(O_F^+, w_\beta) \cong \widetilde{(O_F^+)}_\Comp$$
as topological spaces.
\end{rem}

When $F = \begin{bmatrix}0 & 1 \\ -q^{-1} & 0\end{bmatrix}$, $-1<q\ne 0<1$ we get $O_F^+ \cong SU_q(2)$ as compact quantum groups. We would like to transfer the above result to the setting of $SU_q(2)$ for later use. Let $\varphi : {\rm Pol}(SU_q(2)) \to \mathbb{C}$ be a non-zero character. By examining the commutation relations it is straightforward to check that
$$\varphi(c_q) = \varphi(c^*_q) = 0\;\; \text{and}\;\; \varphi(a_q) = \rho \in \mathbb{C}\backslash\{0\}.$$
Thus, we get
$$||v^{(s)}||_\infty = \max\{\abs{\rho}, \abs{\rho^{-1}}\}^{2s},\;\; s\in \frac{1}{2}\z_+.$$
	
\begin{thm}\label{thm-SUq(2)-char}
We have $\widetilde{SU_q(2)}_\Comp \cong \mathbb{C}\backslash\{0\}$ as topological groups and for the exponential weight $w_\beta$ on the dual of  $SU_q(2)$ as in Example \ref{ex-central-weights} we have
\begin{align*}
{\rm Spec}\, A(SU_q(2), w_\beta)
& \cong \{\rho \in \Comp\backslash\{0\} : \frac{1}{\beta} \le |\rho| \le \beta\}\\
& \cong \{ V = \begin{bmatrix} \rho & 0 \\ 0 & \rho^{-1}\end{bmatrix}\in M_2(\Comp) : ||V||_\infty\le \beta\}
\end{align*}
as topological spaces.
\end{thm}
			
\subsection{The case of free unitary quantum group $U_F^+$}

We recall the definition of the compact quantum group $U_F^+$ from \cite{Ban} for $F\in GL_n(\mathbb{C})$. The $C^*$-algebra $C(U_F^+)$ is the universal $C^*$-algebra with the generators $U = (u_{ij})^n_{i,j=1}$ satisfying the relations 
	\begin{equation}\label{eq-rel-unitary}
	\begin{cases}UU^* = U^*U = I,\\ (F\overline{U}F^{-1})^*F\overline{U}F^{-1} = I\end{cases} \Leftrightarrow \begin{cases}UU^* = U^*U = I,\\ U^tF^*F\overline{U} = F^*F,\end{cases}
	\end{equation}
where $U^t = (u_{ji})^n_{i,j=1}$,  so that we clearly have $C(U_F^+) = C_u(U_F^+)$. The co-multiplication is again  determined by $\Delta(u_{ij}) = \sum^n_{k=1}u_{ik}\otimes u_{kj}.$ The fusion rule for $U_F^+$ is more involved than the one for $A_o(F)$. We have ${\rm Irr}(U_F^+) = \mathbb{F}^+_2$, the free semigroup with two generators $\{g_1, g_2\}$ and with unit $e$. Recall that $\mathbb{F}^+_2$ is equipped with an anti-multiplicative involution $g \mapsto \bar{g}$ defined by
	$$\bar{e} = e,\; \bar{g_1} = g_2,\; \bar{g_2} = g_1.$$
Then, we have
	$$ u^{(g)} \tp u^{(h)} \cong \bigoplus_{g=\alpha \sigma,\; h=\bar{\sigma}\beta} u^{(\alpha\beta)},\; g,h\in \mathbb{F}^+_2.$$
If we focus on the case of $h = g_1$, then the above rule becomes
	$$u^{(g)} \tp u^{(g_1)} \cong \begin{cases} u^{(gg_1)} \oplus u^{(gg^{-1}_2)} & \text{ if $g$ ends with $g_2$}\\ u^{(gg_1)} & \text{ if $g$ ends with $g_1$}\end{cases}.$$
Similarly we have
	$$u^{(g)} \tp u^{(g_2)} \cong \begin{cases} u^{(gg_2)} \oplus u^{(gg^{-1}_1)} & \text{ if $g$ ends with $g_1$}\\ u^{(gg_2)} & \text{ if $g$ ends with $g_2$}\end{cases}.$$
Note that $u^{(g_1)}=U$ and $u^{(g_2)} = \overline{U}$, and that $\{u^{(g_1)}, u^{(g_2)}\}$ is the canonical generating set of ${\rm Irr}(A_u(F))$. The above fusion rule tells us that the corresponding word length function $\tau$ is the same as the word length function of $\mathbb{F}^+_2$ using the generating set $\{g_1, g_2\}$.
\begin{thm}
We have $\widetilde{(U_F^+)}_\Comp \cong \{V\in GL_n(\Comp): V(F^*F)^t = (F^*F)^tV\}$ as topological groups and for the exponential weight $w_\beta$ on the dual of  $U_F^+$ as in Example \ref{ex-central-weights} we have
	\begin{align*}
	\lefteqn{{\rm Spec}\, A(U_F^+, w_\beta)}\\
	& \cong \{ V\in GL_n(\Comp) : V(F^*F)^t = (F^*F)^t V, \;\; ||V||_\infty\le \beta,\; ||V^{-1}||_\infty \le \beta\}
	\end{align*}
as topological spaces.
\end{thm}
\begin{proof}
It is straightforward to see that a non-zero character $\varphi : {\rm Pol}(U_F^+) \to \mathbb{C}$ is in 1-1 correspondence with $V\in GL_n(\Comp)$ satisfying $V(F^*F)^t = (F^*F)^tV$, where $V$ is given by $V = (\varphi(u_{ij}))^n_{i,j=1}$, and that this 1-1 correspondence defines a homeomorphism. Note that $\varphi$ is a $*$-character if and only if $V$ is a unitary. The character $\varphi$ preserves the fusion rules, so that a simple induction on the length $|g|$ of $g$ tells us that
\[
||v^{(g)}||_\infty = ||V||^n_\infty \cdot ||V^{-1}||^m_\infty,
\]
where $n$ and $m$ are the numbers of $g_1$'s and $g_2$'s in $g$, respectively.
Indeed, the cases $|g|=0$ and $|g|=1$ are trivial. Suppose that we have the conclusion for $|g|\le k$, $k\ge 1$, and we consider the case $|g| = k+1$. If we write $g = hh'$ with $|h|=k$, then we have four possible decompositions of $v^{(hh')} \otimes v^{(h^{''})}$ since $h', h^{''} \in \{g_1, g_2\}$. Now we examine all of the four cases with the observations $v^{(e)} = 1$ and $1\le ||V||_\infty \cdot ||V^{-1}||_\infty$ (from the equation $V(F^*F)^tV^{-1} = (F^*F)^t$) to get the conclusion we wanted. Note that the topological considerations can be done in a similar way as in Theorem \ref{thm-free-ortho-spec}.
\end{proof}

\begin{rem}
As in the classical case we can see that
	$$\bigcup_{\beta\ge 1}{\rm Spec}\, A(U_F^+, w_\beta) \cong (\widetilde{U_F^+})_\Comp$$
as topological spaces.
\end{rem}

\subsection{The case of $S^+_n$}\label{subsection-per}

We recall the definition of Wang's free permutation quantum group $S_n^+$, cf.\ \cite{snplus-wang}. The C$^*$-algebra $C(S_n^+)$ is the universal C$^*$-algebra with the generators $U=(u_{ij})_{i,j=1}^n$ satisfying the relations
	\begin{equation}\label{eq-rel-snplus}
	\begin{cases} u_{ij}=u^2_{ij}=u_{ij}^* \quad \mbox{ for } 1\le i,j \le n,\\ \sum_{i=1}^n u_{ij} = 1 = \sum_{i=1}^n u_{ji} \quad\mbox{ for } 1\le j \le n,\end{cases}
	\end{equation}
so that we clearly have $C(S_n^+) = C_u(S_n^+)$.	
The co-multiplication on $C(S_n^+)$ is determined by $\Delta(u_{ij}) = \sum^n_{k=1}u_{ik}\otimes u_{kj}$, $1\le i,j\le n$. From now on we suppose $n\ge 4$. Then the fusion rule for $S_n^+$ is the same as the one of $SO(3)$, which means that ${\rm Irr}(S_n^+) = \z_+$, $u^{(0)} = 1$, $U = (u_{ij})^n_{i,j=1}\cong 1\oplus u^{(1)}$ and
	$$u^{(1)} \tp u^{(s)} \cong u^{(s+1)} \oplus u^{(s)} \oplus u^{(s-1)} \quad \mbox{ for } s\ge 1,$$
cf.\ \cite{ban-snplus}.
Note that due to the second relation in \eqref{eq-rel-snplus}, the defining representation of $S_n^+$ is not irreducible, but contains the trivial representation. But after decomposing it into irreducible representations it contains only one other representation which we denote by $u^{(1)}$ and which generates ${\rm Irr}(S_n^+)$. The associated length function is simply $\tau(s)=s$.

A character $\varphi$ on ${\rm Pol}(S_n^+)$ is uniquely determined by its values on the generators $v_{ij}=\varphi(u_{ij})$, $1\le i,j\le n$. Due to the first set of relations in \eqref{eq-rel-snplus} we get $v_{ij}\in\{0,1\}$, and then the second set of relations in \eqref{eq-rel-snplus} implies that $V=(v_{ij})_{i,j=1}^n$ is a permutation matrix, i.e., a matrix that has exactly one coefficient equal to one in each row and column, and all other coefficients are equal to zero. We see that all characters on $S_n^+$ are hermitian and extend to characters on the universal C$^*$-algebra $C(S_n^+)$. By appealing to \cite[Theorem 3.11]{KalNeu} and the fact that the weight $w: \widehat{S^+_n} \to (0,\infty)$ is bounded below we have the following.
\begin{thm}\label{thm-spec-snplus}
The spectra associated to $S_n^+$ coincide (as topological groups), and we have
$$\widetilde{S_n^+} \cong \widetilde{(S_n^+)}_\Comp \cong {\rm Spec}\, C(S_n^+)\cong S_n,$$
where $S_n$ denotes the $n$-th permutation group. In particular, we have
	$$S_n \cong \widetilde{(S_n^+)}_\Comp \cong {\rm Spec}\, A(S_n^+,w)$$
as topological groups.	
\end{thm}

\section{Beurling-Fourier algebras on $SU_q(2)$ and their finite dimensional representations}\label{sec-finite-dim-rep}

In this section we will investigate the second, more general scenario, namely examining finite dimensional representations on Beurling-Fourier algebras coming from the complexification structure of the underlying compact quantum group under the general philosophy that Beurling-Fourier algebras can {\it detect} the structure of complexification. Instead of developing the general theory we will focus on a special case, namely $SU_q(2)$ whose complexifications is $SL_q(2,\mathbb{C})$. More precisely we begin with an element $\pi \in \text{sp}\,C_0(SL_q(2,\Comp))$ of the C$^*$-algebraic spectrum of $C_0(SL_q(2,\Comp))$, i.e. $\pi: C_0(SL_q(2, \Comp)) \to B(H)$ is an irreducible $*$-representation for some Hilbert space $H$, which naturally requires details of $SL_q(2,\Comp)$. We provide a summary of \cite{podles+woronowicz90, woronowicz91} in the following two subsections.

\subsection{$SL_q(2, \mathbb{C})$, $SU_q(2)$ and $AN_q$-matrices}\label{subsubsec-SLq-matrix}

It is natural to expect to have generators of the $C^*$-algebra $C_0(SL_q(2,\Comp))$ similar to the coordinate functions on $SL(2,\Comp)$. For that purpose the authors of \cite{podles+woronowicz90} introduced the concept of a \emph{$SL_q(2, \mathbb{C})$-matrix} (\cite{podles+woronowicz90}), which is a $2\times 2$-matrix $A = \begin{bmatrix} a & b\\ c&d \end{bmatrix} \in M_2(B(H))$ for some Hilbert space $H$ whose coefficients $a,b,c,d$ satisfy the following 17 commutation relations.
	$$
	\begin{cases}
	ab=qba,\;\, ac=qca,\;\, ad-qbc=I,\\
	bc=cb,\;\, bd=qdb,\;\, cd=qdc,\\
	ca^*=qa^*c,\;\, da^*=a^*d,\;\, da-q^{-1}bc=I, \\
	cb^*=b^*c,\;\, dc^*=q^{-1}c^*d,\;\, cc^*=c^*c,\\
	ba^*=q^{-1}a^*b+q^{-1}(1-q^2)c^*d,\;\, db^*=qb^*d-q(1-q^2)a^*c,\\ aa^*=a^*a+(1-q^2)c^*c,\;\; dd^*=d^*d-(1-q^2)c^*c,\\
		bb^*=b^*b+(1-q^2)(d^*d-a^*a)-(1-q^2)^2c^*c.
		\end{cases}
	$$
When $A$ is a unitary, then the matrix is of the form $\begin{bmatrix} a & -qc^*\\ c&a^* \end{bmatrix}$ and the list of relations collapses down to the ones of the generators $a_q$ and $c_q$ of $SU_q(2)$, namely, $a^*a+c^*c = I = aa^*+q^2c^*c,\;c^*c=cc^*,\;ac=qca,\;ac^*=qc^*a,$
so that we call it a \emph{$SU_q(2)$-matrix}. When $A$ has a zero term in the $(2,1)$-entry, then it is of the form $\begin{bmatrix} a & n \\ 0 & a^{-1} \end{bmatrix}$ and the list of relations collapses down to the following:
	$$aa^*=a^*a,\;\, an=qna,\;\, na^*=q^{-1}a^*n,\;\, nn^*=n^*n+(1-q^2)((a^*a)^{-1}-a^*a).$$
We call $A$ an \emph{$AN_q$-matrix}, if in addition spec $a \subseteq S_q$, where the set $S_q$ is given by
	$$S_q := \begin{cases}\Real_+, & q>0 \\ \{ |q|^xe^{i\pi x}: x\in \Real\} \cup \{0\}, & q<0.\end{cases}$$

In the proof of \cite[theorem 5.1]{podles+woronowicz90} the authors introduced a family of irreducible $AN_q$-matrices as follows: For each $s\in \frac{1}{2}\z_+$ we define operators $a^d_s$ and $n^d_s$ acting on $\ell^2_{2s+1}$ by
	$$\begin{cases} a^d_s e_k = q^{k}e_k,\\ n^d_s e_k = \sqrt{\abs{q}^{-2s} - \abs{q}^{-2k} - \abs{q}^{2k+2} + \abs{q}^{2s+2}}e_{k+1}\end{cases}$$
for $k = -s, -s+1, \cdots, s$, where $\{e_k : k = -s, -s+1, \cdots, s\}$ is the canonical orthonormal basis for $\ell^2_{2s+1}$ and $e_{s+1} = 0$.

It has been noted (\cite[p. 414]{podles+woronowicz90}) that any irreducible $AN_q$-matrix is (upto unitary equivalence) of the form
		\begin{equation}\label{A_s}
		A_s := \begin{bmatrix} a^d_s & n^d_s\\ 0& (a^d_s)^{-1} \end{bmatrix}
		\end{equation}
for some $s\in \frac{1}{2}\z_+$.  We will prove later that
$$A_{\frac{1}{2}}\otimes A_s \cong A_{s+\frac{1}{2}} \oplus A_{s-\frac{1}{2}}$$
for $s\ge \frac{1}{2}$, cf.\ Lemma \ref{lem-technical}.

We consider their (unbounded) direct sums
	$$\begin{cases}a^d := \oplus_{s\in \frac{1}{2}\z_+} a^d_s,\\ n^d: = \oplus_{s\in \frac{1}{2}\z_+} n^d_s\end{cases}$$
and the following four (unbounded) elements
	$\begin{cases}\alpha := a_q \otimes a^d,\\ \beta := a_q\otimes n^d - qc^*_q\otimes (a^d)^{-1},\\ \gamma := c_q\otimes a^d,\\ \delta :=c_q\otimes n^d + a^*_q\otimes (a^d)^{-1}.\end{cases}$
	
The elements $a^d$ and $n^d$ are {\it affiliated} to the $C^*$-algebra $c_0(\widehat{SU_q(2)})$ and the elements $\alpha, \beta, \gamma, \delta$ are {\it affiliated} to $C_0(SL_q(2,\Comp))$. Recall that the concept of ``affiliated element'' of a $C^*$-algebra has been introduced in  \cite{woronowicz91}. We write $b \eta B$ when an element $b$ is affiliated with a $C^*$-algebra $B$.
In the particular case of $B = c_0\text{-}\oplus_{\alpha\in \I}B_\alpha$ for unital $C^*$-algebras $B_\alpha$ we actually have that $b\eta B$ if and only if $b = (b_\alpha)_{\alpha \in \I}$ with $b_\alpha \in B_\alpha$ for all $\alpha \in \I$. Note that the set of all elements affiliated to the above $A$ allows a natural $*$-algebra structure. Note also that any $*$-representation between $C^*$-algebras can be uniquely extended to affiliated elements.
The elements $a^d, n^d$ and $\alpha, \beta, \gamma, \delta$ generate the $C^*$-algebras $c_0(\widehat{SU_q(2)})$ and  $C_0(SL_q(2,\Comp))$, respectively, in the following sense:
For any $AN_q$-matrix $\begin{bmatrix} \tilde{a} & \tilde{n} \\ 0 & \tilde{a}^{-1} \end{bmatrix}$ acting on a Hilbert space $H$ there is a unique nondegenerate $*$-homomorphism $\pi : c_0(\widehat{SU_q(2)}) \to B(H)$ such that $\pi(a^d)=\tilde{a}$ and $\pi(n^d) = \tilde{n}$, where we use the canonical extension of $\pi$ to the affiliated elements. Similarly, for any $SL_q(2,\Comp)$-matrix $\begin{bmatrix} a & b \\ c & d \end{bmatrix}$ acting on a Hilbert space $H$ there is a unique  $*$-homomorphism $\pi : C_0(SL_q(2, \Comp)) \to B(H)$ such that $\pi(\alpha) = a$, $\pi(\beta) = b$, $\pi(\gamma)=c$ and $\pi(\delta) = d$, where we again use the canonical extension of $\pi$ to the affiliated elements. We would like to record the last observation as follows.

\begin{prop}\label{prop-SLq2-matrix}
For any Hilbert space $H$ there is a 1-1 correspondence between a $*$-homomorphism $\pi : C_0(SL_q(2, \Comp)) \to B(H)$ and a $SL_q(2,\Comp)$-matrix $A\in M_2(B(H))$. We have similar correspondences between a $*$-homomorphism $\pi : c_0(\widehat{SU_q(2)}) \to B(H)$ (resp. $\pi : C(SU_q(2)) \to B(H))$) and a $AN_q$-matrix (resp. $SU_q(2)$-matrix) $A\in M_2(B(H))$. These correspondences preserve irreducibility so that we have
	$\frac{1}{2}\z_+ \cong {\rm sp}\,c_0(\widehat{SU_q(2)})$ via the mapping $s \in \frac{1}{2}\z_+ \mapsto A_s$ from \eqref{A_s}.
\end{prop}

\subsection{The Iwasawa decomposition for $SL_q(2,\Comp)$}\label{subsec-Iwasawa}

In \cite[Theorem 1.3]{podles+woronowicz90} the following Iwasawa decomposition has been proved. Recall that two operators $X$, $Y\in B(H)$ doubly commute if $XY = YX$ and $XY^* = Y^*X$.

\begin{thm}\label{thm-Iwasawa} \cite[Theorem 1.3]{podles+woronowicz90}
For any $SL_q(2, \mathbb{C})$-matrix $A \in M_2(B(H))$ we have a unique decomposition $A = A_c A_d$ where $A_c$ is a $SU_q(2)$-matrix and $A_d$ is an $AN_q$-matrix acting on the same Hilbert space $H$. The matrix elements of $A_c$ doubly commute with matrix elements of $A_d$. Moreover, the matrix elements of $A_c$ and $A_d$ belong to the $C^*$-algebra generated by matrix elements of $A$. 
\end{thm}		

For an irreducible representation $\pi \in {\rm sp}\,C_0(SL_q(2,\Comp))$ acting on $H$, the Iwasawa decomposition and the correspondence \eqref{eq-spectrum} can be combined as follows. From \eqref{eq-spectrum} we have uniquely determined $\pi_c \in {\rm sp}\,C(SU_q(2))$ and $\pi_d \in {\rm sp}\,c_0(\widehat{SU_q(2)})$ acting on $H_c$ and $H_d$, respectively, satisfying $\pi = \pi_c \otimes \pi_d$. Let $A = \begin{bmatrix} a & b \\ c & d \end{bmatrix}$ be the $SL_q(2,\mathbb{C})$-matrix corresponding to the representation $\pi$. Then we have
	\begin{align}\label{eq-Irred-Iwasawa}
	A & = (id_{M_2}\otimes \pi)\left(\begin{bmatrix} \alpha & \beta \\ \gamma & \delta \end{bmatrix}\right) \\
	& = (id_{M_2}\otimes \pi)\left(\begin{bmatrix} a_q \otimes a^d & a_q\otimes n^d - qc^*_q\otimes (a^d)^{-1}\\ c_q\otimes a^d & c_q\otimes n^d + a^*_q\otimes (a^d)^{-1} \end{bmatrix}\right)\nonumber \\
	& = \begin{bmatrix} \pi_c(a_q)\otimes 1_{B(H_d)} & - q\pi_c(c^*_q)\otimes 1_{B(H_d)} \\ \pi_c(c_q)\otimes 1_{B(H_d)} &  \pi_c(a^*_q)\otimes 1_{B(H_d)} \end{bmatrix}\nonumber \\
	& \;\;\;\;\;\; \times  \begin{bmatrix} 1_{B(H_c)} \otimes \pi_d(a^d) & 1_{B(H_c)} \otimes \pi_d(n^d) \\ 0 &  1_{B(H_c)} \otimes \pi_d((a^d)^{-1}) \end{bmatrix} \nonumber\\
	& =: A_cA_d \nonumber,
	\end{align}
Since the above $A_c$ and $A_d$ are $SU_q(2)$-matrix and $AN_q(2)$-matrix with doubly commuting entries, we know that $A = A_cA_d$ is actually the Iwasawa decomposition.

We end this subsection with a discussion how we get a homomorphism $\varphi = \varphi_A$ from ${\rm Pol}(SU_q(2))$ into $B(H)$ out of a $*$-homomorphism $\pi = \pi_A : C_0(SL_q(2, \Comp)) \to B(H)$ associated to a $SL_q(2,\Comp)$-matrix $A$ acting on $H$ with the Iwasawa decomposition $A = A_cA_d$. Let $\A$ be the $*$-algebra of all elements affiliated to $C_0(SL_q(2, \Comp))$ and $\A_{\rm hol}$ be the subalgebra of $\A$ generated by $\alpha, \beta, \gamma, \delta$. Then by \cite[Proposition 6.1]{podles+woronowicz90} we have the following linear multiplicative bijection.
	$$Q : \A_{\rm hol} \to {\rm Pol}(SU_q(2))$$
such that $Q(\alpha) = a_q$, $Q(\beta) = -qc^*_q$, $Q(\gamma) = c_q$ and $Q(\delta) = a^*_q$. This allows us to define a homomorphism $\varphi = \varphi_A$ from ${\rm Pol}(SU_q(2))$ into $B(H)$ as follows.
	$$\varphi_A : {\rm Pol}(SU_q(2)) \stackrel{Q^{-1}}{\longrightarrow} \A_{\rm hol} \subseteq \A \stackrel{\pi_A}{\longrightarrow} B(H),$$
where $\pi_A$ denotes the extended $*$-homomorphism. Note that it is straightforward to check that the extended $*$-homomorphism $\pi_A$ is unital. Now we repeat the construction of the homomorphisms $\varphi_c = \varphi_{A_c}$ and $\varphi_d = \varphi_{A_d}$ using the fact that $A_c$ and $A_d$ are also $SL_q(2,\mathbb{C})$-matrices. Moreover, we get the associated elements (as in Section \ref{subsec-discrete})
	$$v_A = v_{\varphi_A},\; v_c = v_{\varphi_{c}},\; v_d = v_{\varphi_{d}} \in \prod_{s\in \IrrQG} (M_{n_s}\otimes B(H)),$$
respectively. Note that we can easily check $\varphi_c$ extends to a $*$-homomorphism from $C(SU_q(2))$ into $B(H)$.

\subsection{Finite dimensional representations which are (completely) bounded on Beurling-Fourier algebras of $SU_q(2)$}

Let us return to an irreducible $*$-representation $\pi: C_0(SL_q(2, \Comp)) \to B(H)$ for some Hilbert space $H$ with $\pi_c \in {\rm sp}\,C(SU_q(2))$ and $\pi_d \in {\rm sp}\,c_0(\widehat{SU_q(2)})$ acting on $H_c$ and $H_d$, respectively, satisfying $\pi = \pi_c \otimes \pi_d$. We consider $SL_q(2,\mathbb{C})$-matrix $A\in M_2(B(H))$ associated to $\pi$ with the Iwasawa decomposition $A=A_c A_d$ as in \eqref{eq-Irred-Iwasawa}. We also have another $SL_q(2,\Comp)$-matrix $A_{\pi_d} = \begin{bmatrix} \pi_d(a^d) & \pi_d(n^d) \\ 0 & \pi_d((a^d)^{-1}) \end{bmatrix}$. We recall the associated homomorphisms $\varphi=\varphi_A$, $\varphi_c=\varphi_{A_c}$, $\varphi_d=\varphi_{A_d}$ and $\varphi_{A_{\pi_d}}$ from ${\rm Pol}(SU_q(2))$ into $B(H)$ with the corresponding elements $v= v_A$, $v_c$, $v_d$ and $v_{\pi_d}$ as in the end of Section \ref{subsec-Iwasawa}.

\begin{prop}\label{prop-discrete-only}
Let $s\in \frac{1}{2}\z_+$ be arbitrary. Then, we have the following.
\begin{enumerate}
\item We have $v^{(s)} = v_c^{(s)} v_d^{(s)} \in M_{2s+1}\otimes B(H).$
\item The element $v_c^{(s)}$ is a unitary.
\item Consider the three transferred homomorphisms
	$$\varphi\circ\mathcal{F},\; \varphi_d\circ\mathcal{F}: c_{00}(\widehat{SU_q(2)}) \to B(H)$$
and $\varphi_{A_{\pi_d}}\circ\mathcal{F}: c_{00}(\widehat{SU_q(2)}) \to B(H_d)$ as in Section \ref{subsec-discrete}. The map $\varphi\circ\mathcal{F}$ extends to a completely bounded map from $A(SU_q(2), w)$ if and only if $\varphi_d\circ\mathcal{F}$ does so if and only if $\varphi_{A_{\pi_d}}\circ\mathcal{F}$ does so. In that case they all have the same cb-norms.
\end{enumerate}

\end{prop}

\begin{proof}
(1) For $s=0$ we clearly have $v^{(0)} = \varphi(1) = \varphi_c(1)\varphi_d(1) = v_c^{(0)}v_d^{(0)}$. The case $s=1/2$ is exactly the Iwasawa decomposition itself, see Theorem \ref{thm-Iwasawa}.

For $s\ge 1$ we can use the fusion rule of $SU_q(2)$ for the conclusion. We check the case $s=1$ as a demonstration. From the fusion rule we have
\begin{align*}
\lefteqn{(id\otimes \varphi)(u^{(1)}) \oplus \varphi(1)}\\
	& = (id\otimes \varphi)(u^{(1)} \oplus 1)\\
	& \cong (id\otimes \varphi)(u^{(1/2)} \tp u^{(1/2)})\\
	& = (id\otimes \varphi)(u^{(1/2)}) \tp (id\otimes \varphi)(u^{(1/2)})\\
	& = (id\otimes \varphi_c)(u^{(1/2)}) (id\otimes \varphi_d)(u^{(1/2)}) \tp (id\otimes \varphi_c)(u^{(1/2)}) (id\otimes \varphi_d)(u^{(1/2)})\\
	& = [(id\otimes \varphi_c)(u^{(1/2)}) \tp (id\otimes \varphi_c)(u^{(1/2)}) ][(id\otimes \varphi_d)(u^{(1/2)}) \tp (id\otimes \varphi_d)(u^{(1/2)}) ]\\
	& \cong (id\otimes \varphi_c)(u^{(1)} \oplus 1) (id\otimes \varphi_d)(u^{(1)} \oplus 1)\\
	& = (id\otimes \varphi_c)(u^{(1)})(id\otimes \varphi_d)(u^{(1)}) \oplus \varphi_c(1)\varphi_d(1)
\end{align*}
Note that we are using the same unitaries for the equivalences so that we actually get the equality
	$$v^{(1)} = (id\otimes \varphi)(u^{(1)}) = (id\otimes \varphi_c)(u^{(1)})(id\otimes \varphi_d)(u^{(1)}) = v_c^{(1)} v_d^{(1)}$$ by looking at the first summand.

\vspace{0.3cm}

(2) This can be checked using the fusion rule as above.

\vspace{0.3cm}

(3) Let us regard $c_{00}(\widehat{SU_q(2)})$ as a subspace of $A(SU_q(2), w)$ via the embedding \eqref{eq-embed-A(G,w)}, then the corresponding cb-norm of $\varphi\circ \F$ is (Proposition \ref{prop-cb-norm-weighted}) given by
	$$\|\varphi\circ \F\|_{cb} = \sup_{s\in \frac{1}{2}\z_+}\frac{\|v^{(s)}\|_{M_{2s+1}\otimes B(H)}}{w(s)}.$$
Since $v_c^{(s)}$ is a unitary for any $s\in \frac{1}{2}\z_+$ we can easily see that $\|\varphi\circ \F\|_{cb} = \|\varphi_d\circ \F\|_{cb}$. Moreover, we have $\varphi_d(x) = I_{B(H_c)}\otimes \varphi_{A_{\pi_d}}(x)$, $x\in {\rm Pol}(SU_q(2))$, so that we have $\|\varphi_d\circ \F\|_{cb} = \|\varphi_{A_{\pi_d}}\circ \F : A(SU_q(2),w) \to B(H_d)\|_{cb}$.

\end{proof}	

The above Proposition \ref{prop-discrete-only} explains why it is enough to focus only on the case when $A$ itself is an irreducible $AN_q$-matrix, so that we can assume that $A$ is of the form of $A_s$ for $s \in \frac{1}{2}\z_+$ from \eqref{A_s} with the associated unital homomorphism
	$$\varphi_s : {\rm Pol}(SU_q(2)) \to M_{2s+1}.$$
Before we proceed to the main result we need to understand more on irreducible $AN_q$-matrices. For two $AN_q$-matrices $A \in M_2(B(H))$ and $B \in M_2(B(K))$ we define $A \bp B \in M_2(B(H\otimes K))$, which is also an $AN_q$-matrix (\cite[Proposition 1.2]{podles+woronowicz90}).

	\begin{lem} \label{lem-technical}
	We have the following.
		\begin{enumerate}
			
			\item For each $s\in \frac{1}{2}\z_+$ with $s\ge 1/2$ we have
				$$A_{1/2} \bp A_s \cong A_{s+1/2} \oplus A_{s-1/2}.$$
			
			\item Let $A \in M_2(B(H))$ and $B \in M_2(B(K))$ be $AN_q$-matrices. Let $\varphi_A$, $\varphi_B$ and $\varphi_{A \bp B}$ be the homomorphisms from ${\rm Pol}(SU_q(2))$ associated to $A$, $B$ and $A \bp B$, respectively with the corresponding elements $v_A$, $v_B$ and $v_{A \bp B}$. Then we have
				$$v_{A \bp B}^{(t)} = v_A^{(t)} \bp v_B^{(t)}\;\;\text{ for each }\;\; t\in \frac{1}{2}\z_+.$$
		\end{enumerate}	
	\end{lem}
\begin{proof}
(1) We need to show the following unitary equivalence.
		\begin{align*}\label{eq-AN-fusion}
		\lefteqn{\begin{bmatrix} {a_{1/2}^d} & {n_{1/2}^d}\\ 0 & ({a_{1/2}^d})^{-1} \end{bmatrix} \bp \begin{bmatrix} {a_s^d} & {n_s^d}\\ 0 & ({a_s^d})^{-1} \end{bmatrix}}\\
		& \cong \begin{bmatrix} {a_{s+1/2}^d} & {n_{s+1/2}^d}\\ 0 & ({a_{s+1/2}^d})^{-1}\end{bmatrix} \oplus \begin{bmatrix} {a_{s-1/2}^d} & {n_{s-1/2}^d}\\ 0 & ({a_{s-1/2}^d})^{-1}\end{bmatrix}.\nonumber
		\end{align*}
First, we recall that spec ${a_s^d} = \{q^{-s}, q^{-s+1}, \cdots, q^s \}$. Since the left hand side is also a $AN_q$-matrix acting on finite dimensional Hilbert space it must be decomposed into a direct sum of the irreducible ones. Thus, the comparison of the spectrum of the $(1,1)$-entry on both sides gives us the conclusion.

\vspace{0.5cm}

(2) We use the fusion rule again for the conclusion. The cases $t=0$ and $t=1/2$ are trivial. We check the case of $t=1$ as a demonstration. From the fusion rule and \eqref{eq-tensor-top-bot} we have
	\begin{align*}
	\lefteqn{v^{(1)}_{A \bp B} \oplus 1}\\
	& = (id\otimes \varphi_{A \bp B})(u^{(1)}) \oplus \varphi_{A \bp B}(1) = (id\otimes \varphi_{A \bp B})(u^{(1)} \oplus 1)\\
	& \cong (id\otimes \varphi_{A \bp B})(u^{(1/2)} \tp u^{(1/2)}) = (id\otimes \varphi_{A \bp B})(u^{(1/2)}) \tp (id\otimes \varphi_{A \bp B})(u^{(1/2)})\\
	& = (A \bp B) \tp (A \bp B) = (A \tp A) \bp (B \tp B)\\
	& = ((id\otimes \varphi_A)(u^{(1/2)}) \tp (id\otimes \varphi_A)(u^{(1/2)})) \bp ((id\otimes \varphi_B)(u^{(1/2)}) \tp (id\otimes \varphi_B)(u^{(1/2)}))\\
	& = ((id\otimes \varphi_A)(u^{(1/2)} \tp u^{(1/2)})) \bp ((id\otimes \varphi_B)(u^{(1/2)} \tp u^{(1/2)}))\\
	& \cong (id\otimes \varphi_A)(u^{(1)} \oplus 1) \bp (id\otimes \varphi_B)(u^{(1)} \oplus 1)\\
	& = (v^{(1)}_A \oplus 1) \bp (v^{(1)}_B \oplus 1) = (v^{(1)}_A \bp v^{(1)}_B) \oplus 1
	\end{align*}
Note that we are using the same unitaries for the equivalences so that we actually get the equality
	$$v^{(1)}_{A \bp B} = v^{(1)}_A \bp v^{(1)}_B$$
by looking at the first summand.
\end{proof}

Now we have the main technical result.
	\begin{thm}\label{thm-main-tech}
	Let $v_s = (v^{(t)}_s)_{t \in \frac{1}{2}\z_+} \in \prod_{t \in \frac{1}{2}\z_+}M_{2t+1} \otimes M_{2s+1}$ be the element associated to the $AN_q$-matrix $A_s$ from \eqref{A_s}. Then, there is a constant $C_q$ depending only on $q$ such that
		$$\abs{q}^{-2st} \le ||v^{(t)}_s||_\infty \le C_q^s\abs{q}^{-2st},\;\; \forall s,t \in \frac{1}{2}\z_+.$$
	\end{thm}
\begin{proof}
From the above Lemma \ref{lem-technical} we have
	$$v^{(t)}_{1/2} \bp v^{(t)}_s \cong v^{(t)}_{s+1/2} \oplus v^{(t)}_{s-1/2}$$
for each $s,t \in \frac{1}{2}\z_+$ with $s\ge 1/2$.	
Now by the fact that the matrix multiplication $M_n \prt M_n \to M_n, \;\; A\otimes B \mapsto AB$ is a complete contraction we have
	$$||v^{(t)}_{1/2} \bp v^{(t)}_s||_\infty \le ||v^{(t)}_{1/2}||_\infty \cdot ||v^{(t)}_s||_\infty.$$
Then the unitary equivalence tells us that
	$$||v^{(t)}_{s+1/2}||_\infty \le ||v^{(t)}_{1/2}||_\infty \cdot ||v^{(t)}_s||_\infty,$$
so that we have
	$$||v^{(t)}_s||_\infty \le ||v^{(t)}_{1/2}||^{2s}_\infty = ||(id \otimes \varphi_{1/2})(u^{(t)})||^{2s}_\infty.$$
From this point on we would like to estimate $||(id \otimes \varphi_{1/2})(u^{(t)})||_\infty$ using the following detailed information on the matrix coefficient of $u^{(t)}$. For $n,m\in \{-t, -t+1, \cdots, t-1, t\}$ we have (\cite[Proposition 5.2]{Koo})
	\begin{align*}
		u^{(t)}_{nm} & = \binom{2t}{t-n}^{\frac{1}{2}}_{q^{-2}}\binom{2t}{t-m}^{-\frac{1}{2}}_{q^{-2}}\sum^{(t-n) \land (t+m)}_{i=0 \lor (m-n)}q^{(t-n-i)(n-m+2i)}q^{-i(n-m+i)}\\
		&\;\;\;\; \times \binom{t-n}{i}_{q^{-2}}\binom{t+n}{t+m-i}_{q^{-2}} (-qc^*_q)^ic^{n-m+i}_qa^{t-n-i}_q(a^*_q)^{t+m-i},
	\end{align*}
where $\binom{a}{b}_q$ are the $q$-binomial coefficients given by
	$$\binom{a}{b}_q = \frac{\prod^a_{j=1}(1-q^j)}{\prod^b_{j=1}(1-q^j) \cdot \prod^{a-b}_{j=1}(1-q^j)}.$$	
Since $\varphi_{1/2}(c_q) = 0$ and $\varphi_{1/2}(-qc^*_q) = n^d_{1/2} = \begin{bmatrix} 0 & 0\\ q^{-\frac{1}{2}}(1-q^2) & 0 \end{bmatrix}$ with $(n^d_{1/2})^2 = 0$, we can easily see that the only surviving terms in the summation of $\varphi_{1/2}(u^{(t)}_{nm})$ are the cases $i = m-n =0$  and $i = m-n = 1$. When $i = m-n =0$ we get
	$$\varphi_{1/2}(u^{(t)}_{nn}) = (a^d_{1/2})^{-2n},$$
where 	$a^d_{1/2} = \begin{bmatrix} q^{-\frac{1}{2}} & 0\\ 0 & q^{\frac{1}{2}} \end{bmatrix}$. When $i = m-n =1$ we have
	\begin{align*}
	\lefteqn{\varphi_{1/2}(u^{(t)}_{n,n+1})}\\
	& = \binom{2t}{t-n}^{\frac{1}{2}}_{q^{-2}}\binom{2t}{t-n-1}^{-\frac{1}{2}}_{q^{-2}} \binom{t-n}{1}_{q^{-2}} q^{t-n-1}n^d_{1/2}(a^d_{1/2})^{-2n-1}\\
	& = \left(\frac{1-q^{-2(t+n+1)}}{1-q^{-2(t-n)}}\right)^{\frac{1}{2}} \cdot \frac{1-q^{-2(t-n)}}{1-q^{-2}} \cdot  q^{t-n-1} \cdot  \begin{bmatrix} q^{-\frac{1}{2}}(1-q^2)q^{n+\frac{1}{2}} & 0\\ 0 & 0 \end{bmatrix}.
	\end{align*}
Thus, we have
	\begin{align*}
	\lefteqn{||(id \otimes \varphi_{1/2})(u^{(t)})||_\infty}\\
	& \le ||\sum_{-l \le n \le l}e_{nn} \otimes \varphi_{1/2}(u^{(t)}_{nn})||_\infty + ||\sum_{-l \le n < l}e_{n,n+1} \otimes \varphi_{1/2}(u^{(t)}_{n,n+1})||_\infty\\
	& = \max_{-l \le n \le l} ||\varphi_{1/2}(u^{(t)}_{nn})||_\infty + \max_{-l \le n < l} ||\varphi_{1/2}(u^{(t)}_{n,n+1})||_\infty\\
	& = C_q |q|^{-t}
	\end{align*}
for some constant $C_q$ depending only on $q$. This gives us the upper bound we wanted.

For the lower bound we consider the restriction down to one of the diagonals. First, we observe that $\varphi_s({\rm Pol}(SU_q(2))) \subseteq L_{2s+1}$, where $L_n$ is the subalgebra of $M_n$ consisting of lower triangular matrices. The advantage of the range of $\varphi_s$ being lower triangular matrices is that the canonical projection
	$$P_{2s+1}: L_{2s+1} \to D_{2s+1}$$
onto the algebra of diagonal matrices $D_{2s+1}$ is multiplicative. Finally, by composing the evaluation functional 
	$$\psi_s: D_{2s+1} \to \Comp,\; (a_k)_{-s \le k \le s} \mapsto a_{-s},$$
which is also multiplicative,  with $\varphi_s$ and $E_{2s+1}$ we get a non-zero character
	$$\Phi_s : {\rm Pol}(SU_q(2)) \stackrel{\varphi_s}{\to} L_{2s+1} \stackrel{E_{2s+1}}{\to} D_{2s+1} \stackrel{\psi_s}{\to} \Comp$$
satisfying $\Phi_s(a_q) = q^{-s}, \Phi_s(a^*_q) = q^s, \Phi_s(c_q) = \Phi_s(c^*_q) = 0$. Then Theorem \ref{thm-SUq(2)-char} gives us the lower bound.
\end{proof}

Combining all the above we get the main result of this section.

\begin{thm}
Let $w_\beta$, $\beta \ge 1$ be the exponential weight on the dual of $SU_q(2)$ from Example \ref{ex-central-weights}, i.e. $w_\beta(s) = \beta^{\tau(s)} = \beta^{2s}$, $s\in \frac{1}{2}\z_+$. Let $\varphi_s$ be the unital homomorphism associated to the $AN_q$-matrix $A_s$, $s\in \frac{1}{2}\z_+$ from \eqref{A_s}. Then, the transferred homomorphism $\tilde{\varphi}_s$ extends to a (completely) bounded map on $A(SU_q(2), w_\beta)$ if and only if $\abs{q}^{-s} \le \beta$. Moreover, we have
\begin{align*}
\lefteqn{{\rm sp}\,C_0(SL_q(2,\Comp))}\\
& \cong {\rm sp}\,C(SU_q(2))\times \bigcup_{\beta\ge 1}\{A_s: s\in \frac{1}{2}\z_+,\; \text{$\tilde{\varphi}_s$ is bounded on $A(SU_q(2), w_\beta)$}\}
\end{align*}
as topological spaces.
\end{thm}

\begin{proof}
Note that
	$$||\tilde{\varphi}_s: A(SU_q(2), w_\beta) \to M_{2s+1} ||_{cb} = \sup_{t\in \frac{1}{2}\z_+}\frac{||v^{(t)}_s||_\infty}{w_\beta(t)} = \sup_{t\in \frac{1}{2}\z_+}\frac{||v^{(t)}_s||_\infty}{\beta^{2t}}.$$
This gives the first assertion with Theorem \ref{thm-main-tech} since the map $\tilde{\varphi}_s$ with the finite dimensional range is completely bounded if and only if it is bounded.

For the second assertion we only need to observe that
$${\rm sp}\,c_0(\widehat{SU_q(2)}) \cong \frac{1}{2}\z_+ \cong \bigcup_{\beta\ge 1}\{A_s: s\in \frac{1}{2}\z_+,\; \text{$\tilde{\varphi}_s$ is bounded on $A(SU_q(2), w_\beta)$}\}$$
by Proposition \ref{prop-SLq2-matrix}.
\end{proof}

\begin{rem}{\rm
\begin{enumerate}
\item
If we apply a similar argument using the fusion rule as in the character case, then we get the upper bound.
$$\norm{v^{(t)}_{\varphi_s}}_\infty \le C^t_q \abs{q}^{-2st}.$$
Note that the above estimate involves $t$ instead of $s$, which is not the conclusion we wanted.
	
\item
In general we do not know the exact formula of $v^{(t)}_{\varphi_s}$ except the case of $s=1/2$. That is the main technical difficulty we have.
\end{enumerate}
}
\end{rem}

\section*{Acknowledgements} We, the authors, are grateful to Robert Yuncken and Christian Voigt for valuable discussions and informations. We also thank to the anonymous referee for his/her careful reading and valuable suggestions.

\bibliographystyle{alpha}

\end{document}